\documentclass[a4paper, DIV=12, 11pt]{amsart}
\usepackage[latin1]{inputenc}
\usepackage{amssymb, amsthm, amsmath, thmtools, mathtools}
\usepackage{amsfonts}
\usepackage[abbrev]{amsrefs} 

\usepackage{graphicx}
\usepackage{thmtools}
\usepackage{hyperref}
\usepackage[bottom, marginal]{footmisc}
\usepackage{todonotes}

\usepackage{tikz}
\usetikzlibrary{matrix}

\declaretheoremstyle[headfont=\normalfont]{normalhead}
\newtheorem{lemma}{Lemma}[section]
\newtheorem{theorem}[lemma]{Theorem}
\newtheorem{proposition}[lemma]{Proposition}
\newtheorem{corollary}[lemma]{Corollary}
\newtheorem{definition}[lemma]{Definition}
\newtheorem{remark}[lemma]{Remark}

\newcounter{mt}

\newtheorem{maintheorem}[mt]{Theorem}
\newtheorem{maincorollary}[mt]{Corollary}

\newtheorem*{acknowledgement}{Acknowledgement}

\newcommand{\R}{\mathbb{R}}
\newcommand{\C}{\mathbb{C}}

\DeclareMathOperator{\Val}{Val}
\DeclareMathOperator{\VConv}{VConv}
\DeclareMathOperator{\Conv}{Conv}
\DeclareMathOperator{\vol}{vol}

\DeclareMathOperator{\supp}{supp}

\DeclareMathOperator{\diam}{diam}

\DeclareMathOperator{\GL}{GL}

\DeclareMathOperator{\GW}{\mathrm{GW}}

\DeclareMathOperator{\Gr}{\mathrm{Gr}}

\DeclareMathOperator{\Sym}{\mathrm{Sym}}
\DeclareMathOperator{\Poly}{\mathcal{P}}

\DeclareMathOperator{\MAVal}{\mathrm{MAVal}}

\DeclareMathOperator{\F}{\mathbb{F}}
\DeclareMathOperator{\MA}{\mathrm{MA}}
\DeclareMathOperator{\Mat}{\mathrm{Mat}}
\DeclareMathOperator{\LV}{\mathrm{LV}}

\renewcommand{\Re}{\operatorname{Re}}
\renewcommand{\Im}{\operatorname{Im}}
\renewcommand{\O}{\mathcal{O}}
\renewcommand{\P}{\mathrm{P}}
\newcommand{\M}{\mathcal{M}}

\newcommand{\A}{\mathrm{A}}
\newcommand{\Aff}{\mathrm{Aff}}

\newcommand{\locsupp}{\mathrm{supp}_{\mathrm{loc}}}

\numberwithin{equation}{section}

\author{Jonas Knoerr}
\title[Integral representation of polynomial local functionals]{Integral representation of polynomial local functionals on convex functions}
\date{}

\newcommand{\Addresses}{{
		\bigskip
		\footnotesize
		
		Jonas Knoerr, \textsc{Institute of Discrete Mathematics and Geometry, TU Wien, Wiedner Hauptstrasse 8-10, 1040 Wien, Austria}\par\nopagebreak
		\textit{E-mail address}: \texttt{jonas.knoerr@tuwien.ac.at}
		
		\medskip
	}}
	
\makeatletter
\def\blfootnote{\xdef\@thefnmark{}\@footnotetext}
\makeatother

\makeindex
\begin{document}
\maketitle
\begin{abstract}
	Integral representations for continuous polynomial local functionals on convex functions are established in terms of a finite family of polynomials. This result is obtained by approximation from a classification of the dense subspace of smooth polynomial local functionals, which is based on a Paley--Wiener--Schwartz-type classification of the Goodey--Weil distributions associated to these functionals under support restrictions. As an application, density results for various families of Monge--Amp\`ere-type operators are established. 
\end{abstract}
\tableofcontents

\section{Introduction}

Consider the space $\M(\R^n)$ of complex Radon measures on $\R^n$. Given a family $\mathcal{F}$ of functions on $\R^n$, we call a map $\Psi:\mathcal{F}\rightarrow\M(\R^n)$ a \emph{local functional} or \emph{locally determined} if for all $f,h\in \mathcal{F}$ such that $f|_U=h|_U$ for some open set $U\subset\R^n$, the associated measures satisfy
\begin{align*}
	\Psi(f)|_U=\Psi(h)|_U.
\end{align*}
The class of local functionals includes many well-known geometric operators, including the Dirichlet energy, a variety of Monge--Amp\`ere-type operators, as well as other classical functionals from the calculus of variations. In applications, these functionals are usually given by an integral representation for sufficiently regular functions but need to be extended by some procedure to a more general function space (with better completeness or compactness properties). Usually, this procedure involves some extension by (semi-) continuity or by approximation by more regular expressions. In both cases, the properties of the extension are only available indirectly, and it is a non-trivial problem to obtain a useful representation of this functional.\\
This naturally leads to the question under which general conditions a local functional on a given function space admits a suitable integral representation, with the goal to establish properties of these representations in terms of properties of the underlying functional. For local functionals on Sobolev spaces, these questions have been investigated extensively over the last 40 years, compare \cite{AlbertiIntegralrepresentationlocal1993,BottaroOppezziMultipleintegralrepresentations1985,ButtazzoDalMasoIntegralrepresentationrelaxation1985,DalMasoEtAlIntegralrepresentationclass1994,EssebeiEtAlIntegralrepresentationlocal2023,MaioneEtAlconvergencefunctionals2020}.\\

This is the second of two articles in which we consider this problem for a certain class of local functionals on convex functions (see \cite{KnoerrPolynomiallocalfunctionals2025} for the first part). Let $\Conv(\R^n,\R)$ denote the space of all finite convex functions on $\R^n$. This space is naturally equipped with the topology induced by epi-convergence, which in this setting coincides with the topology induced by locally uniform or pointwise convergence (compare the references in \autoref{sectionPreliminaries}). We consider $\M(\R^n)$ as the topological dual of the space $C_c(\R^n)$ of all compactly supported continuous functions on $\R^n$ (with respect to its usual inductive topology), and equip $\M(\R^n)$ with the weak* topology. The primary example of a continuous local functional on $\Conv(\R ^n,\R)$ is the real Monge--Amp\`ere operator $\MA:\Conv(\R^n,\R)\rightarrow\M(\R^n)$, which is given by
\begin{align*}
	\MA(f;B)=\int_{B}\det(D^2f(x))dx
\end{align*}
for $f\in\Conv(\R^n,\R)\cap C^2(\R^n)$ and bounded Borel sets $B\subset\R^n$, where $D^2f$ denotes the Hessian of $f$, but extends by continuity to arbitrary convex functions. More generally, one can consider the mixed Monge--Amp\`ere operators
\begin{align*}
	\MA(f_1,\dots,f_n):=\frac{1}{n!}\frac{\partial^n}{\partial\lambda_1\dots\partial\lambda_n}\Big|_0\MA\left(\sum_{j=1}^n\lambda_jf_j\right)
\end{align*} 
for $f_1,\dots,f_n\in\Conv(\R^n,\R)$. Operators of this type have recently become an important tool in geometric valuation theory as a key ingredient in the construction and classification of invariant valuations on convex bodies and functions, compare \cite{AleskerPlurisubharmonicfunctionsoctonionic2008,AleskerValuationsconvexfunctions2019,ColesantiEtAlHadwigertheoremconvex2022,ColesantiEtAlHadwigertheoremconvex2023,ColesantiEtAlHadwigertheoremconvex2024,ColesantiEtAlHadwigertheoremconvex2025,KnoerrSingularValuationsHadwiger2025,KnoerrUnitarilyinvariantvaluations2026,Knoerrgeometricdecompositionunitarily2024,MouamineMussnigVectorialHadwigerTheorem2025}. Here, a map $\mu:\Conv(\R^n,\R)\rightarrow\mathcal{A}$ into an Abelian semi-group is called a valuation if
\begin{align*}
	\mu(f)+\mu(h)=\mu(f\vee h)+\mu(f\wedge h)
\end{align*}
for all $f,h\in\Conv(\R^n,\R)$ such that the pointwise minimum $f\wedge h$ is convex (note that the pointwise maximum $f\vee h$ is always convex). In particular, these classification results rely on the observation that various Monge--Amp\`ere-type operators can be considered as measure-valued valuations on $\Conv(\R^n,\R)$. As pointed out by Alesker in \cite{AleskerValuationsconvexfunctions2019}, this property can be obtained from a result by B{\l}ocki \cite{BlockiEquilibriummeasureproduct2000} for the complex Monge--Amp\`ere operator, which relies on the fact that this operator is locally determined. Recently, the author showed that this applies in broader generality.
\begin{theorem}[\cite{KnoerrPolynomiallocalfunctionals2025}*{Theorem~A}]\label{theorem:LocalFuncValuations}
	Let $\Psi:\Conv(\R^n,\R)\rightarrow\M(\R^n)$ be a continuous local functional. Then $\Psi$ is a valuation.
\end{theorem}
Let us remark that this is in general not the case if the functional is not continuous, compare \cite{KnoerrPolynomiallocalfunctionals2025}*{Example~4.1}. For continuous local functionals, this allows us to apply tools from valuation theory to local functionals. We will be interested in the following class of functionals: Let $\A(n,\R)$ denote the space of all affine maps $\ell:\R^n\rightarrow\R$. We call a local functional $\Psi$ on $\Conv(\R^n,\R)$ \emph{polynomial of degree at most $d\in\mathbb{N}$} if the map
\begin{align*}
	\A(n,\R)&\rightarrow \M(\R^n)\\
	\ell&\mapsto \Psi(f+\ell)
\end{align*}
is a polynomial of degree at most $d$ in $\ell\in\A(n,\R)$ for every $f\in\Conv(\R^n,\R)$. In this case, we also call $\Psi$ a polynomial local functional for brevity without specifying the degree. This notion is closely related to translation invariance for valuations on polytopes and convex bodies, compare \cite{Aleskermultiplicativestructurecontinuous2004,PukhlikovKhovanskiiFinitelyadditivemeasures1992}. Related invariance properties for local functionals on Sobolev spaces have been considered in \cite{AlbertiIntegralrepresentationlocal1993,ButtazzoDalMasocharacterizationnonlinearfunctionals1985,ButtazzoDalMasoIntegralrepresentationrelaxation1985}.\\
Let $\P_d\LV(\R^n)$ denote the space of all continuous local functionals on $\Conv(\R^n,\R)$ that are polynomial of degree at most $d\in\mathbb{N}$. In \cite{KnoerrPolynomiallocalfunctionals2025}*{Theorem~B}, the relation to valuations on convex functions was used to establish a homogeneous decomposition for this space:
\begin{align*}
	\P_d\LV(\R^n)=\bigoplus_{k=0}^{n+d}\P_d\LV_k(\R^n),
\end{align*}
where $\Psi\in \P_d\LV_k(\R^n)$ if and only if $\Psi$ is $k$-homogeneous, i.e. if $\Psi(tf)=t^k \Psi(f)$ for $t\ge 0$ and $f\in\Conv(\R^n,\R)$. It is then not difficult to see that elements of $\P_d\LV_0(\R^n)$ are constant local functionals, and thus this space is isomorphic to $\M(\R^n)$. For the elements belonging to the top degree component $\P_d\LV_{d+n}(\R^n)$, unique integral representations in terms of the real Monge--Amp\`ere operator were established in \cite{KnoerrPolynomiallocalfunctionals2025}*{Theorem~C}, which provides a complete classification of these functionals. In this article, we establish integral representations for the intermediate degrees in terms of a finite family of polynomials depending on the partial derivatives of a given function up to order $2$ related to the minors of the Hessian. Let $\Poly(V)$ denote the space of complex-valued polynomials on a finite dimensional real vector space $V$ and $\Poly_d(V)$ the subspace of all polynomials of degree at most $d$. Consider the space $\mathrm{M}_n\subset \Poly(\Sym^2(\R^n))$ spanned by all $k$-minors, $0\le k\le n$, of a symmetric $(n\times n)$-matrix, where the $0$-minor of a matrix is $1$ by definition.\\
Note that the group $\Aff(n,\R)$ of all invertible affine transformations of $\R^n$ acts on $\P_d\LV(\R ^n)$ by
\begin{align}
	\label{eq:DefAction}
	[\pi(g)\Psi](f;B)=\Psi(f\circ g;g^{-1}(B))
\end{align}
for $g\in\Aff(n,\R)$ and $\Psi\in\P_d\LV(\R^n)$, $f\in\Conv(\R^n,\R)$, $B\subset \R^n$ bounded Borel set. We denote the subspace of all translation invariant local functionals with respect to this action by $\P_d\LV(\R^n)^{tr}$. These admit the following characterization in terms of polynomials. Let us identify $\R\times(\R^n)^*\cong\A(n,\R)$ using the map $(c,y)\mapsto [x\mapsto \langle y,x\rangle+c]$.
\begin{theorem}[\cite{KnoerrPolynomiallocalfunctionals2025}*{Theorem~D}]\label{theorem:ClassificationPolyTranslationInv}
	For every $\Psi\in\P_d\LV(\R^n)^{tr}$ there exists a unique polynomial $P_\Psi\in \Poly_d(\A(n,\R))\otimes \mathrm{M}_n$ such that
	\begin{align*}
		\Psi(f;B)=\int_{B}P_\Psi(f(x),df(x),D^2f(x))dx
	\end{align*}
	for all $f\in \Conv(\R^n,\R)\cap C^2(\R^n)$, $B\subset\R^n$ bounded Borel set. Conversely, the right hand side of this equation extends by continuity to a unique element in $\P_d\LV(\R^n)^{tr}$ for every polynomial in $\Poly_d(\A(n,\R))\otimes \mathrm{M}_n$.
\end{theorem}
In particular, $\P_d\LV(\R^n)^{tr}$ is a finite dimensional space, and we set $N(n,d):=\dim\P_d\LV(\R^n)^{tr}$. Using these polynomials, we establish the following unique integral representations for these functionals for sufficiently regular functions.
\begin{maintheorem}
\label{maintheorem:IntegralRepContinuousCase}
	Let $\Psi_j$, $1\le j\le N(n,d)$, be a basis for $\P_d\LV(\R^n)^{tr}$. For every $\Psi\in \P_d\LV(\R^n)$ there exist unique measures $\mu_j\in \M(\R^n)$ such that for every $f\in\Conv(\R^n,\R)\cap C^2(\R^n)$, $B\subset\R^n$ bounded Borel set,
	\begin{align}
		\label{eq:IntegralRepresentationContCase}
		\Psi(f;B)=\sum_{j=1}^{N(n,d)}\int_{B} P_{\Psi_j}(f(x),df(x),D^2f(x))d\mu_j(x).
	\end{align}
\end{maintheorem}
Note that \autoref{maintheorem:IntegralRepContinuousCase} provides a representation of these functionals but not a complete characterization since the left hand side of Eq.~\eqref{eq:IntegralRepresentationContCase} does not extend to an element of $\P_d\LV(\R^n)$ for arbitrary measures $\mu_j\in\M(\R^n)$. In general, these measures are not absolutely continuous with respect to the Lebesgue measure for a given local functional, and determining their properties seems to be a highly non-trivial problem.\\

We will obtain \autoref{maintheorem:IntegralRepContinuousCase} by approximation from a characterization of a certain subspace of $\P_d\LV(\R^n)$. In \cite{KnoerrPolynomiallocalfunctionals2025} three different topologies on $\P_d\LV(\R^n)$ were considered, which correspond to the compact-open topology with respect to three different topologies on $\M(\R^n)$. We review these results in \autoref{section:topologies} but will focus on one of them for this introduction. We consider $\P_d\LV(\R^n)$ with the compact-open topology with respect to the weak* topology on $\M(\R^n)$, i.e., we equip $\P_d\LV(\R^n)$ with the topology induced by the semi-norms defined for $\phi\in C_c(\R^n)$, $K\subset\Conv(\R^n,\R)$ compact by
\begin{align*}
	\Psi\mapsto\sup_{f\in K}|\Psi(f;\phi)|,
\end{align*}
where $\phi\mapsto \Psi(f;\phi)$ denotes the integration functional on $C_c(\R^n)$ induced by the measure $\Psi(f)$. This turns $\P_d\LV(\R^n)$ into a locally convex vector space and we call this topology the \emph{compact-to-weak* topology}. The action of $\Aff(n,\R)$ on $\P_d\LV(\R^n)$ given by Eq.~\eqref{eq:DefAction} defines a representation such that the map
\begin{align*}
	\Aff(n,\R)&\rightarrow\P_d\LV(\R^n)\\
	g&\mapsto \pi(g)\Psi
\end{align*}
is continuous for every $\Psi\in \P_d\LV(\R^n)$. We will call $\Psi\in \P_d\LV(\R^n)$ \emph{smooth} if the restriction of this map to the subgroup of translations is smooth with respect to the compact-to-weak* topology. Let $\P_d\LV^\infty(\R^n)$ denote the subspace of all smooth local functionals. Then $\P_d\LV^\infty(\R^n)$ is dense in $\P_d\LV(\R^n)$ with respect to the compact-to-weak* topology, compare the discussion in \autoref{section:topologies}. We establish the following characterization of smooth polynomial local functionals.
\begin{maintheorem}\label{maintheorem:IntegralRepSmooth}
	Let $\Psi_j$, $1\le j\le N(n,d)$, be a basis for $\P_d\LV(\R^n)^{tr}$. A local functional $\Psi\in \P_d\LV(\R^n)$ is smooth if and only if there exist smooth functions $\phi_j\in C^\infty(\R^n)$ such that
	\begin{align*}
		\Psi(f;B)=\sum_{j=1}^{N(n,d)} \int_{B}\phi_j(x)d\Psi_j(f;x)
	\end{align*} 
	for all $f\in\Conv(\R^n,\R)$, $B\subset\R^n$ bounded Borel set. In this case, the functions $\phi_j$ are uniquely determined by $\Psi$ and the chosen basis.
\end{maintheorem}
\autoref{maintheorem:IntegralRepSmooth} implies that the subspace of smooth local functionals is a free $C^\infty(\R^n)$-module of rank $N(n,d)$ with respect to the $C(\R^n)$-module structure on $\P_d\LV(\R^n)$ defined by associating to $\psi\in C(\R^n)$ and $\Psi\in \P_d\LV(\R^n)$ the local functional $\psi\bullet \Psi$ given by
\begin{align*}
	[\psi\bullet \Psi](f;B)=\int_B \psi(x)d\Psi(f;x)
\end{align*}
for $f\in\Conv(\R^n,\R)$, $B\subset \R^n$ bounded Borel set. In \autoref{section:affineInvSubModule}, we use \autoref{maintheorem:IntegralRepSmooth} to establish a correspondence between the closed $C(\R^n)$-submodules of $\P_d\LV(\R^n)$ and the subspace of $\P_d\LV(\R^n)^{tr}$. For $d=0$, this has the following interesting consequence.
\begin{maintheorem}
	\label{maintheorem:affineInvariantSubmoduleDense}
	Let $W\subset \P_0\LV_k(\R^n)$ be an affine invariant $C(\R^n)$-submodule. If $W\ne 0$, then $W$ is dense in $\P_0\LV_k(\R^n)$ with respect to the compact-to-weak* topology. 
\end{maintheorem}
This result implies a variety of density results for different families of local functionals. We include the following as an example.
\begin{maincorollary}\label{maincorollary:ExampleDenseSubmodule}
	Let $\mathcal{F}\subset \Conv(\R^n,\R)$ be affine invariant. Then the $C(\R^n)$-submodule of $\P_0\LV(\R^n)$ generated by the mixed Monge-Amp\`ere operators
	\begin{align*}
		\MA(\cdot[k],f_1,\dots,f_{n-k}),\quad f_1,\dots,f_{n-k}\in\mathcal{F}
	\end{align*}
	is either trivial or dense in $\P_0\LV_k(\R^n)$ with respect to the compact-to-weak* topology.
\end{maincorollary}
Some interesting families are listed at the end of \autoref{section:affineInvSubModule}.

\subsection{Plan of the article}
	The central result of this article is the classification of smooth polynomial local functionals in \autoref{maintheorem:IntegralRepSmooth}. The key step is the case $d=0$, from which the case $d\ge 1$ can be obtained by induction. Our approach is based on an interpretation of $k$-homogeneous polynomial local functionals in terms of certain distributions introduced in \cite{KnoerrPolynomiallocalfunctionals2025}. By using a partition of unity, we reduce the proof of \autoref{maintheorem:IntegralRepSmooth} to a characterization of certain compactly supported distributions. More precisely, we follow an approach introduced in \cite{KnoerrPaleyWienerSchwartz2025} that characterizes such distributions by showing that their Fourier--Laplace transforms belong to certain modules of holomorphic functions generated by a finite family of polynomials. In our case, it turns out that these polynomials are minors of certain matrices, and we use an adaptation of the method in \cite{KnoerrPaleyWienerSchwartz2025} to obtain the representation in \autoref{maintheorem:IntegralRepSmooth} from the decomposition of the Fourier--Laplace transform in terms of these generators. In particular, we obtain a version of the Paley--Wiener--Schwartz Theorem for homogeneous and compactly supported local functionals that are polynomial of degree $0$, see \autoref{section:PWS}.\\
	
	The article is structured as follows:\\
	In \autoref{sectionPreliminaries}, we investigate certain modules of polynomials generated by $k$-minors. The results are simple adaptations of \cite{KnoerrPaleyWienerSchwartz2025}*{Section 2.3}, however, we include the necessary modifications for the convenience of the reader. We also review the required results for valuations and local functionals, including the construction of the Goodey--Weil distributions introduced in \cite{KnoerrPolynomiallocalfunctionals2025}. \autoref{section:PWS} contains a characterization of these distributions for $d=0$ in terms of the Fourier--Laplace transform and a version of \autoref{maintheorem:IntegralRepSmooth} for compactly supported local functionals. The general case is obtained from this description in \autoref{section:IntegralRepresentations}, which also includes the proof of \autoref{maintheorem:IntegralRepContinuousCase}. Finally, we prove \autoref{maintheorem:affineInvariantSubmoduleDense} in \autoref{section:affineInvSubModule}.
	
	\begin{acknowledgement}
		This research was funded in whole or in part by the Austrian Science Fund (FWF), \href{https://www.doi.org/10.55776/PAT4205224}{10.55776/PAT4205224}.
	\end{acknowledgement} 

\section{Preliminaries}
	\label{sectionPreliminaries}
	Throughout the article, we will use the following constants:
	\begin{align*}
		N(n,d)=&\dim \P_d\LV(\R^n)^{tr},\\
		N_{n,k}=&\dim \P_0\LV_k(\R^n).
	\end{align*}
	For simplicity, all vector spaces are assumed to be complex vector spaces unless stated otherwise. In this case, for a real vector space $V$, we denote by $V_\C:=V\otimes \C$ its complexification. Abusing notation, we denote by $\Poly(W)$ the space of complex polynomials on a complex vector space $W$. With the notation in the introduction, this implies that for a real vector space $V$, $\Poly(V)=\Poly(V_\C)$. 

	\subsection{$k$-minors and modules of holomorphic functions}
		\label{section:minors}
		As discussed in the introduction, the proof of \autoref{maintheorem:IntegralRepSmooth} is based on a characterization of a certain distribution associated to a given local functional in terms of the Fourier--Laplace transform. In the cases we will be interested in, the relevant distributions are compactly supported, so the Fourier--Laplace transforms define entire functions on suitable products of $\C^n$. In this section we establish the results needed for a description of these functions in Section \ref{section:PWS}.\\
		
		For $0\le k\le n$, we identify $(\C^n)^{k+1}\cong \Mat_{n,k+1}(\C)$ with the space of complex $(n\times(k+1))$-matrices. In particular, for $w\in \Mat_{n,k+1}(\C)$ we will write $w=(w_1,\dots,w_{k+1})$ for its $k+1$ columns $w_1,\dots,w_{k+1}\in\C^n$. In order to be consistent with the usual notation for the entries of matrices, the components of the column vector $w_j\in \C^n$ will be denoted by $w_j=(w_{1j},\dots,w_{nj})^T$.\\
		
		Let $\Poly(\Mat_{n,k+1}(\C))$ denote the space of complex polynomials on $\Mat_{n,k+1}(\C)$. We consider this space as a module over the space $\Poly(\C^n)$ of complex polynomials on $\C^n$ by identifying $\Poly(\C^n)$ with polynomials in $\Poly(\Mat_{n,k+1}(\C))$ depending on the last columns $w_{k+1}$ only. Similarly, we identify the space $\Poly(\Mat_{n,k}(\C))$ of polynomials on the space of $(n\times k)$-matrices with the subset of $\Poly(\Mat_{n,k+1}(\C))$ of polynomials depending on the first $k$ columns only. Let $\M^2_k\subset \Poly(\Mat_{n,k}(\C))$ denote the subspace spanned by quadratic products of $k$-minors and $\M_k\subset\Poly(\Mat_{n,k}(\C))$ the subspace spanned by the $k$-minors.\\
		\begin{remark}
			\label{remark:dimensionM2k}
			It was shown in \cite{KnoerrMongeAmpereoperators2024}*{Corollary 6.19} (see also \autoref{lemma:QBijectiveMA} below) that $\M^2_k\cong \P_0\LV_k(\R^n)^{tr}$. In particular, $\dim \M^2_k=N_{n,k}$.
		\end{remark}
		
		The next two results characterize the submodule $\widetilde{\M^2_k}$ of $\Poly(\Mat_{n,k+1}(\C))$ generated by $\M^2_k$ in terms of the behavior under restrictions. Note that for any $k$-dimensional complex subspace $F\subset \C^n$, we may consider $F^k\subset (\C^n)^k$ as a subset of $\Mat_{n,k}(\C)$. If we choose a basis of $F$, then the restriction of any $k$-minor to $F^k$ is a multiple of the determinant of the coordinate matrix of $w\in F^k$ with respect to the given basis. In particular, the restriction of elements in $\M_k$ to $F^k$ defines a $1$-dimensional subspace in $\Poly(F^k)$.\\
		Let $\prescript{\C}{}{\Gr_k(\C^n)}$ denote the space of $k$-dimensional complex subspaces of $\C^n$ and consider for $\Delta\in \M_k$ the set
		\begin{align*}
			U_\Delta:=\{E\in\ \prescript{\C}{}{\Gr_k(\C^n)}: \Delta|_{E^k}\ne 0\}.
		\end{align*}
		Note that this is a Zariski dense subset of $\prescript{\C}{}{\Gr_k(\C^n)}$ unless $\Delta$ vanishes identically. In the rest of this section, we consider polynomials that restrict to certain multiples of $\Delta^2$ on subspaces belonging to $U_\Delta$. The arguments are very similar to the discussion in \cite{KnoerrPaleyWienerSchwartz2025}*{Section~2.2}, but we include the necessary modifications for completeness. The first lemma shows that it is sufficient to consider restrictions to complexifications of real subspaces.
		\begin{lemma}
			\label{lemma:PolynomialFromComplexifiedSpaces}
			Let $P\in\Poly(\Mat_{n,k+1}(\C))$ be a polynomial with the following property: For every $\Delta\in 	\M_k$ and every $E\in U_\Delta\cap \{E_0\otimes \C: E_0\in\Gr_k(\R^n)\}$ there exists a polynomial $P_{\Delta,E}\in \Poly(E)$ such that for all $w_1,\dots,w_k\in E$, $w_{k+1}\in\C^n$,
			\begin{align}
				\label{eq:restrictionRealSubspaces}
				P(w_1,\dots,w_{k+1})=\Delta^2(w_1,\dots,w_k) P_{\Delta,E}(w_{k+1}).
			\end{align}
			Then there exist polynomials $Q_{\Delta,E}\in \mathcal{P}(E)$ for all $E\in U_\Delta$ such that \begin{align*}
				P(w_1,\dots,w_k)=\Delta^2(w_1,\dots,w_k) Q_{\Delta,E}(w_{k+1}).
			\end{align*}
			for all $w_1,\dots,w_{k}\in E$, $w_{k+1}\in\C^n$.
		\end{lemma}
		\begin{proof}
			For $\Delta=0$, the claim is vacuously true, so we may assume that $\Delta\ne 0$. We may further assume that $P$ is a homogeneous polynomial. Set $W_\Delta:=\{w\in \Mat_{n,k}(\C):\Delta(w)\ne 0\}$ and consider the regular function on $W_\Delta\times\C^n$ given by
			\begin{align*}
				Q(w_1,\dots,w_{k+1})=\frac{P(w_1,\dots,w_{k+1})}{\Delta^2(w_1,\dots,w_k)}.
			\end{align*}
			We claim that the value of $Q(w_1,\dots,w_{k+1})$ only depends on $w_{k+1}$ and $\mathrm{span}(w_1,\dots,w_k)$. In order to see this, it is sufficient to show that
			\begin{align*}
				P(w_1,\dots,w_{k+1})\Delta^2(w_1',\dots,w_k')-
				P(w'_1,\dots,w'_k,w_{k+1})\Delta^2(w_1,\dots,w_k)
			\end{align*}
			vanishes if $\mathrm{span}(w_1,\dots,w_k)=\mathrm{span}(w'_1,\dots,w'_k)$. For $c=(c_{jl})_{jl}\in\C^{k\times k}$, consider the polynomial
			\begin{align*}
				\tilde{Q}(c,w_1,\dots,w_{k+1}):=&P(w_1,\dots,w_{k+1})\Delta^2\left(\sum_{l=1}^kc_{1l}w_l,\dots,\sum_{l=1}^kc_{kl}w_l\right)\\
				&-
				P\left(\sum_{l=1}^kc_{1l}w_l,\dots,\sum_{l=1}^kc_{kl}w_l,w_{k+1}\right)\Delta^2(w_1,\dots,w_k).
			\end{align*}
			If $w_1,\dots,w_{k}\in\R^n$, then $c\mapsto \tilde{Q}(c,w_1,\dots,w_{k+1})$ vanishes for $c\in \R^{k\times k}$ due to Eq.~\eqref{eq:restrictionRealSubspaces}, and thus for every $c\in \C^{k\times k}$. Therefore, the map $(w_1,\dots,w_{k+1})\mapsto \tilde{Q}(c,w_1,\dots,w_{k+1})$ vanishes for $w_1,\dots,w_{k}\in\R^n$ for every $c\in \C^{k\times k}$, $w_{k+1}\in\C^n$, so it vanishes identically, which shows the claim.\\
			In order to construct the polynomial $Q_{\Delta,E}\in\Poly(\C^n)$ for any $E\in U_\Delta$, fix a basis $w^E_1,\dots,w^E_k$ of $E$. By the previous discussion, the value of
			\begin{align*}
				Q_{\Delta,E}(w_{k+1}):=Q(w^E_1,\dots,w^E_k,w_{k+1})=\frac{P(w^E_1,\dots,w^E_k,w_{k+1})}{\Delta^2(w^E_1,\dots,w^E_k)}
			\end{align*}
			only depends on the subspace $E$ and $w_{k+1}\in\C^n$ and not the choice of the basis $w^E_1,\dots,w^E_k$. Since this is a polynomial in $w_{k+1}$, this concludes the proof.
		\end{proof}
		The following completes the characterization of $\widetilde{\M^2_k}$ in terms of its restrictions. The proof uses some basic result on regular representations of $\GL(n,\C)$, and we refer to \cite{GoodmanWallachSymmetryrepresentationsinvariants2009} for a general background.
		\begin{proposition}
			\label{proposition:CharacterizationModuleSquareMinors}
			Let $P\in \Poly(\Mat_{n,k+1}(\C))$ be a polynomial with the following property:
			For every $\Delta\in \M_k$ and every subspace $E\in U_\Delta$ there exists a polynomial $P_{\Delta,E}\in \Poly(\C^n)$ such that
			\begin{align}
				\label{eq:PropositionRestrictionMinors} 
				P(w_1,\dots,w_{k+1})=\Delta^2(w_1,\dots,w_k) P_{\Delta,E}(w_{k+1})
			\end{align}
			for all $w_1,\dots,w_k\in E$, $w_{k+1}\in\C^n$.
			Then $P$ is contained in the $\Poly(\C^n)$-submodule $\widetilde{\M^2_k}$.
		\end{proposition}
		\begin{proof}
			The proof of this result is essentially identical to the proof of \cite[Proposition 2.4]{KnoerrPaleyWienerSchwartz2025}. For completeness, we include the proof with the necessary modifications.\\
			
			Note that the polynomials with these properties form a $\GL(n,\C)$-invariant subspace of $\Poly(\Mat_{n,k+1}(\C))$, where $g\in\GL(n,\C)$ operates on $\Poly(\Mat_{n,k+1}(\C))$ via left multiplication by
			\begin{align*}
				[g\cdot P](w)=P(g^Tw),\quad w\in\Mat_{n,k+1}(\C).
			\end{align*}
			Similarly, the submodule $\widetilde{\M^2_k}$ is a $\GL(n,\C)$-invariant subspace. As the space of complex polynomials on $\Mat_{n,k+1}(\C)$ is a regular representation of $\GL(n,\C)$, it decomposes into a direct sum of irreducible representations. Any such representation is generated by a unique highest weight vector, so it is sufficient to show that any highest weight vector satisfying \eqref{eq:PropositionRestrictionMinors} is contained in the module generated by $\M^2_k$. We will show that any such vector is a product of the square $\Delta_k^2$ of the $k$th principal minor $\Delta_k$ with a polynomial in $\Poly(\C^n)$.\\
			
			First note the claim is trivial for $k=n$, since $\M^2_n$ is spanned by the square of the determinant. Thus assume that $0\le k\le n-1$.\\
			
			Let $P_\lambda$ be a highest weight vector with weight $\lambda\in\mathbb{Z}^n$ and the properties above. Then $P_\lambda$ is invariant under the group $N_n^+$ of upper triangular matrices with $1$ on the diagonal, and for any diagonal matrix $h=\mathrm{diag}(h_1,\dots,h_n)$ we have
			\begin{align*}
				h\cdot P_\lambda=h^\lambda P_\lambda.
			\end{align*}
			Let us consider the restriction of $P_\lambda$ to the dense open set of all elements $w\in \Mat_{n,k+1}(\C)$ with $\Delta_i(w)\ne 0$ for $1\le i\le k+1$, where $\Delta_i$ denotes the $i$th principal minor. Using the Gauss decomposition (compare e.g. \cite[Chapter~11.6]{GoodmanWallachSymmetryrepresentationsinvariants2009}), we can write any such matrix as $w=\nu  h u$, where $\nu\in N^-_n$ is a lower triangular $(n\times n)$-matrix with $1$ on the diagonal, $h\in D_{n,k+1}$ is a diagonal $(n\times (k+1))$-matrix, and $u\in N^+_{k+1}$. Thus
			\begin{align*}
				P_\lambda(w)= P_\lambda((\nu^T)^Thu)=P_\lambda(hu),
			\end{align*}
			since $\nu^T\in N^+_n$.	In particular, $P_\lambda$ is uniquely determined by its restriction to upper triangular matrices. If $w=\nu U$ for an upper triangular matrix $U=(u_1,\dots,u_{k+1})\in \Mat_{n,k+1}(\C)$ and $\nu\in N^-_n$, then $u_1,\dots,u_k\in \C^{k}\times\{0\}$ belong to a $k$-dimensional subspace, and we obtain
			\begin{align*}
				P_\lambda(w)=P_\lambda((\nu^T)^T\cdot U)=P_\lambda(U)=\Delta^2_k(U)Q(u_{k+1})
			\end{align*} 
			for the polynomial $Q:=P_{\Delta_k,\C^k\times\{0\}}|_{\C^{k+1}\times \{0\}}$ by assumption. Note that we are assuming that $u_{k+1}\in \C^{k+1}\times \{0\}\cong \C^{k+1}$. We will thus consider $Q$ as a polynomial on $\C^{k+1}$. In particular, since $\Delta_k$ is invariant under the operation of $N_n^+$,
			\begin{align*}
				P_\lambda(w)=\Delta_k^2(w)Q(u_{k+1}),
			\end{align*}
			and so $P_\lambda$ is uniquely determined by the polynomial $Q$. Since we assume that $P_\lambda$ is non-trivial, $Q$ does not vanish identically. Let us examine the polynomial $Q$ on $\C^{k+1}\times\{0\}$ as well as the weight $\lambda$. First note that for any $u_{k+1}\in \C^{k+1}\times\{0\}$ with $(u_{k+1})_{k+1}\ne 0$ there exist an upper triangular matrix $w_0\in \Mat_{n,k+1}(\C)$ with $\Delta_i(w_0)\ne0$ for $1\le i\le k+1$ such that $u_{k+1}$ is the last column of $w_0$. Given such a matrix $w_0\in\Mat_{n,k+1}(\C)$ and a diagonal matrix $h\in D_{n,n}$, we thus have
			\begin{align*}
				h^{\lambda} P_\lambda(w_0)=P_\lambda(h^T\cdot w_0)=h_1^2\dots h_k^2 \Delta^2_k(w_0)Q(\mathrm{diag}(h_1,\dots,h_{k+1})u_{k+1}),
			\end{align*}
			where we consider $u_{k+1}$ as an element of $\C^{k+1}\cong \C^{k+1}\times \{0\}$. If $u_{k+1}\in \C^{k+1}\times\{0\}\cong\C^{k+1}$ belongs to the complement of the zero set of $Q$, then the right hand side of this equation does not depend on $h_{k+2},\dots,h_n$, so  $\lambda_{k+2}=\dots=\lambda_n=0$ since this set is dense in $\C^{k+1}$ and $\Delta_k(w_0)\ne 0$. Moreover, this implies that $Q$ is a weight vector with weight $(\lambda_1,\dots,\lambda_{k+1})-(2,\dots,2,0)\in \mathbb{Z}^{k+1}$ of the $\GL(k+1,\C)$-representation $\Poly(\C^{k+1})$ with action 
			\begin{align*}
				g\cdot P(z)=P(g^Tz)\quad \text{for}~g\in\GL(k+1,\C),~P\in \Poly(\C^{k+1}).
			\end{align*} If $\nu\in N^+_{k+1}\subset \GL(k+1,\R)$ is an upper triangular unidiagonal matrix, then, since $P_\lambda$ and $\Delta_k$ are invariant under $N^+_n$,
			\begin{align*}
				\Delta_k^2(w_0)Q(u_{k+1})=P_\lambda(w_0)=P_\lambda\left(\begin{pmatrix}
					\nu &0\\
					0 & Id_{n-k-1}
				\end{pmatrix}^Tw_0\right)=\Delta^2_k(w_0)Q(\nu^T u_{k+1}).
			\end{align*}
			In other words, $Q$ is invariant under $N^+_{k+1}$, so $Q$ is a sum of highest weight vectors with weight $(\lambda_1-2,\dots,\lambda_k-2,\lambda_{k+1})$. As $\mathcal{P}(\C^{k+1})$ is a multiplicity free representation of $\GL(k+1,\C)$ and all highest weights are of the form $(d,0,\dots,0)$ for $d\ge0$, we thus obtain $c\in \C\setminus\{0\}$ such that  $Q(z)=cz_1^{d}$ for $z\in \C^{k+1}$. Comparing the degrees of homogeneity, $d'=d-2$. In particular, $(\lambda_1,\dots,\lambda_{k+1})=(d+2,2,\dots,2,0)$, and 
			\begin{align*}
				P_\lambda(w)=c\cdot\Delta_k^2(w) w_{1,k+1}^{d}
			\end{align*}
			which belongs to the $\Poly(\C^n)$-submodule generated by $\M^2_k$. 
		\end{proof}
		Let us consider the space $\mathcal{O}_{\C^m\times\C^n}$ of all entire function on $\C^m\times\C^n$ as a $\O_{\C^n}$-module using the obvious inclusion $\O_{\C^n}\subset \O_{\C^m\times\C^n}$. The next result provides a decomposition of entire functions in $\mathcal{O}_{\C^m\times\C^n}$ that belong to certain $\O_{\C^n}$-submodules generated by homogeneous polynomials. It exploits the power series expansion of these function in $0$, which we will call its power series expansion for brevity. \\		
		We similarly consider $\Poly(\C^m\times\C^n)$ as a module over $\Poly(\C^n)$ using the obvious inclusion, which turns $\Poly(\C^m\times\C^n)$ into a $\Poly(\C^n)$-module. If $M\subset \Poly(\C^m\times\C^n)$ is a finitely generated submodule, then $M$ belongs to the finitely generated free module $F$ generated by the monomials of a set of generators of $M$. We fix an order on the monomials of this free module by restricting the lexicographic order on $\Poly(\C^m\times\C^n)$ induced from the variable order
		\begin{align*}
			w_{1,1}>\dots>w_{m,1}>w_{1,2}>\dots >w_{n,2}
		\end{align*}
		to $F$, where $w_1=(w_{1,1},\dots,w_{m,1})\in \C^m$ and $w_2=(w_{1,2},\dots,w_{n,2})\in \C^n$ denote the coordinates of $w(w_1,w_2)\in\C^m\times\C^n$, and denote the initial term of an element $P\in M$ (i.e. the multiple of the largest monomial in $P$ with non-trivial coefficient) by $\mathrm{in}(P)$. The following result was established in \cite{KnoerrPaleyWienerSchwartz2025} using some elementary properties of Gröbner basis. We refer to \cite{EisenbudCommutativealgebra1995}*{Chapter 15} for a general background. Let us call a $\Poly(\C^n)$-submodule of $\Poly(\C^m\times\C^n)$ a homogeneous module if it is generated by homogeneous polynomials. 
		\begin{theorem}[\cite{KnoerrPaleyWienerSchwartz2025} Theorem 2.9 and Corollary 2.10]
			\label{theorem:DivisionAlg}
			Let $M\subset \mathcal{P}(\C^m\times\C^n)$ be a finitely generated homogeneous module over $\Poly(\C^n)$ and $P_1,\dots,P_N\in M$ an ordered Gröbner basis consisting of homogeneous elements. Assume that this basis has the property that no multiple of $P_i-\mathrm{in}(P_i)$ contains a monomial that is divisible by the initial term of $P_i$ for every $1\le i\le N$. Then the following holds:
			\begin{enumerate}
				\item If $F\in\mathcal{O}_{\C^m\times\C^n}$ is a function such that every homogeneous term in its power series expansion belongs to $M$, then there exist functions $g_j\in \mathcal{O}_{\C^n}$ such that $F=\sum_{i=1}^Ng_iP_i$.
				\item The functions in (1) can be chosen such that for every $\delta>0$ there exists a constant $C_\delta>0$ depending on $P_1,\dots,P_N$ and $\delta$ only, such that 
				\begin{align*}
					|g_j(z)|\le& C_\delta (1+|z|)^{N(h_2+n)}\sup_{\zeta\in 	D_{\delta}(0,z)}|F(\zeta)|,
				\end{align*}
				where $(h_1,h_2)\in\mathbb{N}^2$ is a bound on the degree of $P_j$ for $1\le j\le N$. 
			\end{enumerate} 
		\end{theorem}
		\begin{remark}\label{remark:PropertyGroebnerSatisfied}
			If $M\subset \mathcal{P}(\C^m\times\C^n)$ is a $\Poly(\C^n)$-submodule generated by a finite set of homogeneous polynomials in $\Poly(\C^m)$, then any minimal Gröbner basis of $M$ satisfies the condition on the Gröbner basis in the previous theorem. This applies in particular to the submodule $\widetilde{\M}^2_k$ of $\Poly(\Mat_{n,k+1}(\C))$.
		\end{remark}
		For $\delta>0$ and $w\in \Mat_{n,k+1}(\C)$, set
		\begin{align*}
			D_\delta(w)=\{\zeta\in \Mat_{n,k+1}(\C):&|\Re(\zeta_j)|_1\le |\Re(w_j)|_1+\delta,\\
			&|\Im(\zeta_j)|_1\le |\Im(w_j)|_1+\delta,~1\le j\le k+1\}.
		\end{align*}
		
		\begin{theorem}
			\label{theorem:GroebnerAlgModuleMinor}
			Let $F\in \mathcal{O}_{\Mat_{n,k+1}(\C)}$ be a function such that every homogeneous term in the power series expansion of $F$ belongs to the $\Poly(\C^n)$-module $\widetilde{\M}^2_k$. For every basis $Q_j$, $1\le j\le N_{n,k}$, of $\M^2_k$ there exist functions $g_{j}\in\mathcal{O}_{\C^n}$ such that 
			\begin{align}
				\label{eq:presentationTildeM2k}
				F(w)=\sum_{j=1}^{N_{n,k}}g_{j}(w_{k+1})Q_j(w_1,\dots,w_k).
			\end{align}
			Moreover, these functions can be chosen such that for every $\delta>0$, 
			\begin{align*}
				|g_{j}(z)|\le C_{\delta}\left(1+|z|\right)^{nN_{n,k}}\sup_{\zeta\in D_\delta(0,\dots,0,z)}|F(\zeta)|,
			\end{align*}
			for $N_{n,k}=\dim \M^2_k$ and suitable constants $C_\delta>0$ depending on the chosen basis and $\delta>0$, $n$, and $k$ only.
		\end{theorem}
		\begin{proof}
			Obviously the claim holds for any basis as soon as it holds for some basis. If we choose a basis of $\M^2_k$ that is a Gröbner basis for the module $\widetilde{\M^2_k}$, then we can directly apply \autoref{theorem:DivisionAlg} due to \autoref{remark:PropertyGroebnerSatisfied}. Note that $\dim\M^2_k=N_{n,k}$, compare \autoref{remark:dimensionM2k}.
		\end{proof}
		This implies the following estimate.
		\begin{corollary}
			\label{corollary:EstimateFourierDiagonalCoordinates}
			For every $\delta>0$ there exists a constant $C_\delta>0$ such that every $F\in \O_{\C^n}\widetilde{\M}^2_k$ satisfies
			\begin{align*}
				|F(w)|\le C_\delta \left(\prod_{j=1}^{k}|w_j|^2\right)\left(1+|w_{k+1}|\right)^{nN_{n,k}}\sup_{\zeta\in D_\delta(0,\dots,0,w_{k+1})}|F(\zeta)|.
			\end{align*}
		\end{corollary}
		\begin{proof}
			Choose a presentation
			\begin{align*}
				F(w)=\sum_{j=1}^{N_{n,k}}g_{j}(w_{k+1})Q_j(w_1,\dots,w_k)
			\end{align*}
			such that $g_j\in\mathcal{O}_{\C^n}$ satisfies the estimate in \autoref{theorem:GroebnerAlgModuleMinor}. Then 
			\begin{align*}
				|F(w)|\le \sum_{j=1}^{N_{n,k}}|g_{j}(w_{k+1})||Q_j(w_1,\dots,w_k)|,
			\end{align*}
			and the claim follows by bounding $|Q_j(w_1,\dots,w_k)|$, which is homogeneous of degree $2$ in each argument, by a multiple of $\prod_{j=1}^{k}|w_j|^2$.
		\end{proof}
	\subsection{Polynomial valuations and local functionals on convex functions}
		In this section, we recall the necessary background for polynomial valuations on convex functions and their relation to local functionals.\\
		Recall that we call a map $\mu:\Conv(\R^n,\R)\rightarrow F$ into a locally convex vector space  $F$ a valuation if
		\begin{align*}
			\mu(f\vee h)+\mu(f\wedge h)=\mu(f)+\mu(h)
		\end{align*}
		for all $f,h\in\Conv(\R^n,\R)$ such that their pointwise minimum $f\wedge h$ belongs to $\Conv(\R^n,\R)$. As for local functionals, we call a valuation $\mu:\Conv(\R^n,\R)\rightarrow F$ a polynomial valuation of degree at most $d\in \mathbb{N}$ if the map
		\begin{align*}
			\A(n,\R)&\rightarrow F\\
			\ell&\mapsto \mu(f+\ell)
		\end{align*} 
		is a polynomial of degree at most $d$ for every $f\in\Conv(\R^n,\R)$. Equivalently, there exist unique functions $Y_j:\Conv(\R^n,\R)\rightarrow F\otimes \Sym^j(\A(n,\R)^*)_\C$, $0\le j\le d$, such that for $f\in\Conv(\R^n,\R)$, $\ell\in\A(n,\R)$,
		\begin{align}\label{eq:translativePolynomialExpansion}
			\mu(f+\ell)=\sum_{j=0}^dY_i(f)[\ell],
		\end{align}
		where we identify the symmetric product $\Sym^j(\A(n,\R)^*)_\C$ with the space of $j$-homogeneous complex-valued polynomials on $\A(n,\R)\cong \R\times(\R^n)^*$.\\
		 Let $\P_d\VConv(\R^n,F)$ denote the space of all continuous valuations $\mu:\Conv(\R^n,\R)\rightarrow F$ that are polynomial of degree at most $d$.		By \cite{KnoerrUlivelliPolynomialvaluationsconvex2026}*{Theorem~1.3}, we have a homogeneous decomposition
		\begin{align}\label{eq:HomDecompValuations}
			\P_d\VConv(\R^n,F)=\bigoplus_{k=0}^{n+d}
			\P_d\VConv_k(\R^n,F),
		\end{align}
		where $\P_d\VConv_k(\R^n,F)$ denotes the subspace of $k$-homogeneous valuations, i.e., all valuations $\mu\in\P_d\VConv(\R^n,F)$ such that $\mu(tf)=t^k\mu(f)$ for all $f\in\Conv(\R^n,\R)$, $t\ge 0$.
		\begin{remark}
			Due to \autoref{theorem:LocalFuncValuations}, every continuous local functional on $\Conv(\R^n,\R)$ is a valuation. In particular, we may consider $\P_d\LV(\R^n)$ as a subspace of $\P_d\VConv(\R^n,\M(\R^n))$. Then the decomposition in Eq.~\eqref{eq:HomDecompValuations} corresponds precisely to the decomposition $\P_d\LV(\R^n)=\bigoplus_{k=0}^{n+d}\P_d\LV_k(\R^n)$.
		\end{remark}
		The homogeneous decomposition in Eq.~\eqref{eq:HomDecompValuations} behaves well with respect to the decomposition in Eq.~\eqref{eq:translativePolynomialExpansion}. More precisely, we have the following result.
		\begin{proposition}[\cite{KnoerrUlivelliPolynomialvaluationsconvex2026}*{Proposition 3.2}]\label{proposition:translativeDecompCompatibleHomDecomp}
			For $\mu\in \P_d\VConv_k(\R^n,F)$ let $Y_j:\Conv(\R^n,\R)\rightarrow F\otimes\Sym^j(\A(n,\R)^*)_\C$ be the unique functions satisfying Eq.~\eqref{eq:translativePolynomialExpansion}. Then $Y_j\in \P_{d-j}\VConv_{k-j}(\R^n,F\otimes \Sym^j(\A(n,\R)^*)_\C)$.
		\end{proposition}
		Given $\mu\in\P_d\VConv_k(\R^n,F)$, we may define its polarization $\bar{\mu}:\Conv(\R^n,\R)^k\rightarrow F$ by
		\begin{align*}
			\bar{\mu}(f_1,\dots,f_k)=\frac{1}{k!}\frac{\partial^k}{\partial\lambda_1\dots\partial\lambda_k}\Big|_0\mu\left(\sum_{j=1}^k\lambda_jf_j\right)
		\end{align*}
		for $f_1,\dots,f_k\in\Conv(\R^n,\R)$. Then $\bar{\mu}$ is additive in each argument (compare \cite[Corollary 3.5]{KnoerrUlivelliPolynomialvaluationsconvex2026} and \cite[Section 4.2]{Knoerrsupportduallyepi2021}), which can be used to extend $\bar{\mu}$ to a distribution. Let us denote the completion of a locally convex vector space $F$ by $\bar{F}$ and its topological dual by $F'$. 
		\begin{theorem}[\cite{KnoerrUlivelliPolynomialvaluationsconvex2026}*{Theorem 1.4 and Corollary 4.8}]
			\label{theorem:GW}
			Let $F$ be a locally convex vector space. For every $\mu\in\P_d\VConv_k(\R^n,F)$ there exists a unique distribution $\GW(\mu):C^\infty_c((\R^n)^k)\rightarrow \bar{F}$ such that for every $\lambda\in F'$, the distribution $\lambda(\GW(\mu))$ has compact support and satisfies
			\begin{align*}
				\lambda(\GW(\mu))[f_1\otimes\dots\otimes f_k]=\lambda\left(\bar{\mu}(f_1,\dots,f_k)\right)
			\end{align*}
			for all $f_1,\dots,f_k\in\Conv(\R^n,\R)\cap C^\infty(\R^n)$.
		\end{theorem}
		We call the distribution $\GW(\mu)$ the Goodey--Weil distribution of $\mu$ due to its relation to a similar construction by Goodey and Weil \cite{GoodeyWeilDistributionsvaluations1984}.
		\begin{remark}
			As pointed out in \cite{KnoerrPolynomiallocalfunctionals2025}*{Remark~2.7}, $\GW(\mu)[\phi_1\otimes\dots\otimes \phi_k]\in F$ for all $\phi_1,\dots,\phi_k\in C^\infty_c(\R^n)$. In particular, $\GW(\Psi)[\phi_1\otimes\dots\otimes \phi_k]$ is a measure for any element $\Psi\in\P_d\LV_k(\R^n)$.
		\end{remark}
		For elements in $\P_d\LV_k(\R^n)$, the following version of this construction will be more useful.
		\begin{proposition}[\cite{KnoerrPolynomiallocalfunctionals2025}*{Corollary~4.11}]\label{definition:ExtendedGW}
			Let $1\le k\le n+d$. For every $\Psi\in \P_d\LV_k(\R^n)$ there exists a unique distribution $\widehat{\GW}(\Psi)$ on $(\R^n)^{k+1}$ such that
			\begin{align*}
				\widehat{\GW}(\Psi)[\phi_1\otimes\dots\otimes \phi_{k+1}]=\int_{\R^n}\phi_{k+1}d\left(\GW(\Psi)[\phi_1\otimes\dots\otimes \phi_{k}]\right).
			\end{align*}
		\end{proposition}
		\begin{remark}
			For $\Psi\in \P_d\LV_0(\R^n)$, we set $\widehat{\GW}(\Psi):=\Psi(0)\in\M(\R^n)$.
		\end{remark}
		The relation between $\widehat{\GW}(\Psi)$ and $\GW(\Psi)$ for $\Psi\in \P_d\LV_k(\R^n)$ may also be stated in the following way: Consider the scalar-valued valuation $\Psi[\phi_{k+1}]\in \P_d\VConv_k(\R^n,\C)$ given by
		\begin{align*}
			(\Psi[\phi_{k+1}])(f)=\int_{\R^n} \phi_{k+1}d\Psi(f).
		\end{align*}
		Then
		\begin{align*}
			\widehat{\GW}(\Psi)[\phi_1\otimes\dots\otimes\phi_{k+1}]=\GW(\Psi[\phi_{k+1}])[\phi_1\otimes\dots\otimes\phi_k].
		\end{align*}
		The support of the distributions above induces a notion of support for polynomial valuations and local functionals. Let $\Delta_k:\R^n\rightarrow(\R^n)^k$ be the diagonal embedding. 
		\begin{definition}\label{definition:support}
			For $\mu\in\P_d\VConv(\R^n,F)$, let $\mu=\sum_{k=0}^{n+d}\mu_k$ be its homogeneous decomposition. Then the support of $\mu$ is the set
			\begin{align*}
				\supp\mu:=\bigcup_{k=1}^{n+d}\Delta_k^{-1}(\supp\GW(\mu_k)).
			\end{align*}
			For $\Psi\in\P_d\LV(\R^n)$, let $\Psi=\sum_{k=0}^{n+d}\Psi_k$ be its homogeneous decomposition. Then the local support of $\Psi$ is the set
			\begin{align*}
				\locsupp\Psi:=\bigcup_{k=0}^{n+d}\Delta_{k+1}^{-1}(\supp\widehat{\GW}(\Psi_k)).
			\end{align*}
		\end{definition}
		\begin{remark}
			If $1\le k\le n+d$ and $\Psi\in \P_d\LV_k(\R^n)$, then $\locsupp\Psi=\supp\Psi$, where we consider $\Psi$ as an element of $\P_d\VConv_k(\R^n,\M(\R^n))$, compare \cite{KnoerrPolynomiallocalfunctionals2025}*{Corollary~4.14.}, so these notions only differ for $0$-homogeneous local functionals.
		\end{remark}
		We have the following characterization of these two notions.
		\begin{proposition}[\cite{KnoerrUlivelliPolynomialvaluationsconvex2026}*{Proposition~4.9} and \cite{KnoerrPolynomiallocalfunctionals2025}*{Corollary~4.17}]\label{proposition:characterizationSupport}\noindent
			\begin{enumerate}
				\item The support of $\mu\in \P_d\VConv(\R^n,F)$ is the unique minimum (with respect to inclusion) among all closed sets $A\subset \R^n$ with the following property: If $f,h\in \Conv(\R^n,\R)$ are two functions with $f\equiv h$ on a neighborhood of $A$, then $\mu(f)=\mu(h)$.
				\item For $\Psi\in\P_d\LV(\R^n)$, the set $\locsupp\Psi$ is the unique minimum among all closed sets $A\subset\R^n$ that satisfy the following two properties:
				\begin{enumerate}
					\item For every $f\in\Conv(\R^n,\R)$, $\supp\Psi(f)\subset A$.
					\item If $f,h\in\Conv(\R^n,\R)$ satisfy $f\equiv h$ in a neighborhood of $A$, then $\Psi(f)=\Psi(h)$.
				\end{enumerate}
			\end{enumerate}
		\end{proposition}
		Recall that $\P_d\LV(\R^n)$ is a $C(\R^n)$-module with respect to the module structure defined for $\Psi\in \P_d\LV(\R^n)$, $\phi\in C(\R^n)$ by
		\begin{align*}
			(\phi\bullet\Psi)(f;B)=\int_{B}\phi d\Psi(f)
		\end{align*}
		for $f\in\Conv(\R^n,\R)$, $B\subset\R^n$ bounded Borel set. In other words, for $\psi\in C_c(\R^n)$, 
		\begin{align*}
			(\phi\bullet\Psi)(f;\psi)=\Psi(f;\phi\cdot\psi).
		\end{align*}
		The characterization of the local support in \autoref{proposition:characterizationSupport} directly implies the following compatibility property.
		\begin{corollary}\label{corollary:supportModuleProduct}
			For $\Psi\in\P_d\LV(\R^n)$ and $\phi\in C(\R^n)$, $\locsupp(\phi\bullet\Psi)\subset \supp\phi\cap \locsupp\Psi$.
		\end{corollary}

		The following is proved in \cite[Lemma 5.1]{MussnigVolumepolarvolume2019} for translation invariant valuations, however, the proof only uses the weaker property stated below.
		\begin{lemma}
			\label{lemma:injectivity_characteristic_function}
			Let $(G,+)$ be an Abelian semi-group with cancellation law that carries a Hausdorff topology, and $\mu_1,\mu_2:\Conv(\R^n,\R)\rightarrow G$ two continuous valuations. If $\mu_1(h_P(\cdot-y)+c)=\mu_2(h_P(\cdot-y)+c)$ for all polytopes $P\in\mathcal{K}(\R^n)$ with $0\in \mathrm{int} P$, $y\in \R^n$ and $c\in\R$, then $\mu_1\equiv \mu_2$ on $\Conv(\R^n,\R)$.
		\end{lemma}
		
		This implies the following refined result for polynomial valuations of degree $0$. For a $k$-dimensional subspace $E\in \Gr_k(\R^n)$ let $\pi_E:\R^n\rightarrow E$ denote the orthogonal projection. Given a convex body $K$ in $\R^n$, we denote by $h_K\in\Conv(\R^n,\R)$ its support function.
		\begin{proposition}
			\label{proposition:RestrictionHomValuationLowerDimensionalSupportFunctions}
			Let $F$ be a locally convex vector space and assume that $\mu\in\P_0\VConv_k(\R^n,F)$ has the property that for every $E\in\Gr_k(\R^n)$ and all $K_E\in \mathcal{K}(E)$, $x_E\in E$, $\mu(\pi_E^{*}(h_{K_E}(\cdot-x_E)))=0$. Then $\mu=0$.
		\end{proposition}
		\begin{proof}
			Since $\mu\in\P_0\VConv_k(\R^n,F)$ vanishes if $\lambda\circ \mu=0$ for every $\lambda\in F'$, it is sufficient to consider scalar-valued valuations. In this case, it follows from \cite[Theorem 3]{KnoerrUnitarilyinvariantvaluations2026} that $\mu$ vanishes identically if and only if $\mu(\pi_E^*f_E)=0$ for all $f_E\in \Conv(E,\R)$ and every $k$-dimensional subspace $E\in\Gr_k(\R^n)$. Due to Lemma \ref{lemma:injectivity_characteristic_function}, this is the case if and only if $\mu(\pi_E^{*}(h_{K_E}(\cdot-x_E)))=0$ for all $K\in\mathcal{K}(E)$ and $x_E\in E$.
		\end{proof}
		
	\subsection{Topologies on $\P_d\LV(\R^n)$}
		\label{section:topologies}
		If $F$ is a locally convex vector space, then we equip $\P_d\VConv(\R^n,F)$ with the compact-open topology, i.e. with the family of semi-norms
		\begin{align*}
			\|\mu\|_{F;K}:=\sup_{f\in K}|\mu(f)|_F
		\end{align*}
		where $|\cdot|_F$ is a continuous semi-norm on $F$ and $K\subset\Conv(\R^n,\R)$ is compact (see \cite{Knoerrsupportduallyepi2021}*{Proposition~2.4.} for a description of these sets).\\
		For local functionals, the topology induced by considering $\P_d\LV(\R^n)$ as a subspace of $\P_d\VConv(\R^n,\M(\R^n))$, where $\M(\R^n)$ is equipped with the weak* topology, is exactly the compact-to-weak* topology. However, there are two other standard topologies on $\M(\R^n)$ that lead to two additional topologies on $\P_d\LV(\R ^n)$\\
		Recall that we consider the space $C_c(\R^n)$ of all compactly supported, continuous, $\C$-valued functions on $\R^n$ as a topological vector space with respect to the inductive topology induced by the inclusion of the subspaces $C_A(\R^n)$ of all such functions with support contained in a compact set $A\subset\R^n$ equipped with the topology of uniform convergence. In other words, a subset $U\subset C_c(\R^n)$ is open if and only if the set $U\cap C_A(\R^n)$ is open with respect to the $\|\cdot\|_\infty$-norm. Then it is easy to see that a set $B\subset C_c(\R ^n)$ is bounded if and only if there exits a compact set $A\subset \R^n$ such that $B\subset C_A(\R ^n)$ is bounded. In particular, $\M(\R^n)=C_c(\R^n)'$ can also be considered with the topologies induced by uniform convergence on relatively compact subsets, or the topology induced by uniform convergence on bounded subsets (which is usually called the strong topology). \\		
		It was shown in \cite{KnoerrPolynomiallocalfunctionals2025}*{Lemma~3.2} that any $\Psi\in \P_d\LV(\R^n)$ is uniformly bounded on bounded subsets of $C_c(\R^n)$ on any compact subset of $\Conv(\R^n,\R)$, i.e., for any bounded set $B\subset C_c(\R^n)$ and compact $K\subset\Conv(\R^n,\R)$,
		\begin{align*}
			\|\Psi\|_{(B;K)}:=\sup_{f\in K,\phi\in B}|\Psi(f;\phi)|<\infty
		\end{align*}
		for any $\Psi\in\P_d\LV(\R^n)$ (note, however, that $\Psi$ is in general not continuous with respect to the strong topology on $\M(\R^n)$). In particular, we may take different families $\mathcal{B}$ of bounded sets in $C_c(\R^n)$ and consider the topology induced by the semi-norms $\|\cdot\|_{(B;K)}$ for $B\in\mathcal{B}$ and $K\subset\Conv(\R^n,\R)$ compact. In addition to the family of finite subsets (which induces the uniform-to-weak* topology), the following cases were considered in \cite{KnoerrPolynomiallocalfunctionals2025}.
		\begin{definition}\label{definition:topologies}
			\begin{enumerate}
				\item If $\mathcal{B}$ is the family of all relatively compact subsets, we call the induced topology the compact-to-compact topology.
				\item If $\mathcal{B}$ is the family of all bounded subsets, we call the induced topology the compact-to-bounded topology.
			\end{enumerate}
		\end{definition}
		The compact-to-bounded topology also admits the following description. Let us denote by $B_1(0)$ the unit ball in $\R^n$.
		\begin{proposition}[\cite{KnoerrPolynomiallocalfunctionals2025}*{Proposition~4.18}]
			\label{proposition:semiNormslocalFunctional}
				For $A\subset\R^n$ compact and convex with non-empty interior and $\Psi\in\P_d\LV(\R^n)$, set
			\begin{align*}
				\|\Psi\|_{A,\delta}:=\sup\left\{|\Psi(f;\phi)|: \supp\phi\subset A,\|\phi\|_\infty\le 1, \sup_{x\in A+\delta B_1(0)}|f(x)|\le1\right\}.
			\end{align*}
			Then $\|\cdot\|_{A,\delta}$ is a continuous semi-norm on $\P_d\LV(\R^n)$ with respect to the compact-to-bounded topology and the topology generated by these semi-norms for all  compact and convex $A\subset\R^n$ with non-empty interior and a fixed $\delta>0$ coincides with the compact-to-bounded topology. Moreover, for $0<s<t$ and $\Psi\in\P_d\LV_k(\R^n)$,
			\begin{align}
				\label{eq:relationLocalSeminorms}
				\|\Psi\|_{A,t}\le \|\Psi\|_{A,s}\le \left(\frac{2}{s}(2t+\diam A)+1\right)^k\|\Psi\|_{A,t}.
			\end{align}
		\end{proposition}
		Eq.~\eqref{eq:relationLocalSeminorms} will be relevant for the estimates in \autoref{section:FourierLaplace}. Using this description, it is not difficult to see that $\P_d\LV(\R^n)$ is actually a Fr\'echet space with respect to the compact-to-bounded topology, compare \cite{KnoerrPolynomiallocalfunctionals2025}*{Theorem F}.\\
		
		We refer to \cite{KnoerrPolynomiallocalfunctionals2025}*{Section~3} for a more in depth discussion of the differences between the three topologies and their properties. In this article, we are mostly interested in their compatibility with the action of the group $\Aff(n,\R)$ of affine transformations of $\R^n$. Recall that we defined an action $\pi$ of $\Aff(n,\R)$ on $\P_d\LV(\R^n)$ by
		\begin{align*}
			\pi(g)\Psi(f;\phi)=\Psi(f\circ g;\phi\circ g)
		\end{align*}
		for $g\in\Aff(n,\R)$, $\Psi\in\P_d\LV(\R ^n)$, $f\in\Conv(\R^n,\R)$, and $\phi\in C_c(\R^n)$.

		\begin{proposition}[\cite{KnoerrPolynomiallocalfunctionals2025}*{Lemma~3.12 and Proposition~3.13}]
			\label{proposition:AffActionSeparatelyContinuous}
			The map
			\begin{align*}
				\Aff(n,\R)\times \P_d\LV(\R^n)&\rightarrow \P_d\LV(\R^n)\\
				(g,\Psi)&\mapsto \pi(g)\Psi
			\end{align*} 
			is separately continuous in the compact-to-weak* topology. It is continuous in the compact-to-compact topology.
		\end{proposition}
		For $\Psi\in \P_d\LV(\R^n)$, the map
		\begin{align*}
			\Aff(n,\R)&\rightarrow \P_d\LV(\R^n)\\
			g&\mapsto \pi(g)\Psi
		\end{align*}
		is in general not continuous if $\P_d\LV(\R^n)$ is equipped with the compact-to-bounded topology, compare the example in \cite{KnoerrPolynomiallocalfunctionals2025}*{Corollary~4.20}. This difficulty is the main reason why we have to consider multiple topologies. Fortunately, the space of smooth local functionals is in a sense independent of the chosen topology, as the following special case of \cite{KnoerrPolynomiallocalfunctionals2025}*{Corollary~3.17} shows. Recall that we identify $\R^n$ with the subgroup of $\Aff(n,\R)$ consisting of translations.
		\begin{proposition}
			\label{proposition:EquivalenceNotionsSmoothness}
			For $\Psi\in \P_d\LV(\R^n)$, consider the map
			\begin{align}
				\label{eq:orbitmap}
				\begin{split}
					\R^n&\rightarrow \P_d\LV(\R^n)\\
					x&\mapsto \pi(x)\Psi.
				\end{split}
			\end{align}
			The following are equivalent.
			\begin{enumerate}
				\item This map is smooth in the compact-to-weak* topology.
				\item This map is smooth in the compact-to-compact topology.
				\item This map is smooth in the compact-to-bounded topology.
			\end{enumerate}
		\end{proposition}
		Recall that we denote by the space of smooth local functionals in $\P_d\LV(\R^n )$ by $\P_d\LV^\infty(\R^n)$.
		If the map in Eq.~\eqref{eq:orbitmap} is continuous with respect to the compact-to-bounded topology, then we will call $\Psi$ \emph{strongly continuous}. We denote the space of strongly continuous elements in $\P_d\LV(\R^n)$ by $\P_d\LV^0(\R^n)$. Similarly, we denote by $\P_d\LV_k^0(\R^n)$ the set of strongly continuous and $k$-homogeneous local functionals. It is then easy to see that $\P_d\LV^0(\R^n)=\bigoplus_{k=0}^{n+d}\P_d\LV^0_k(\R^n)$.
		\begin{proposition}[\cite{KnoerrPolynomiallocalfunctionals2025}*{Proposition 3.19. and Corollary 3.20}]
			\label{proposition:PropertiesStronglyContVectors}
			The space $\P_d\LV_k^0(\R^n)$ is closed in $\P_d\LV(\R^n)$ with respect to the compact-to-bounded topology. Moreover, the map
			\begin{align*}
				\R^n\times \P_d\LV^0_k(\R^n)&\rightarrow \P_d\LV^0_k(\R^n)\\
				(x,\Psi)&\mapsto \pi(x)\Psi
			\end{align*}
			is continuous with respect to the compact-to-bounded topology.
		\end{proposition}
		The following is a direct consequence of \cite{KnoerrPolynomiallocalfunctionals2025}*{Corollary~3.22}
		\begin{proposition}\label{proposition:densitySmoothVectors}
			Let $W\subset \P_d\LV_k(\R^n)$ be a translation invariant subspace.
			\begin{enumerate}
				\item If $W$ is closed with respect to the compact-to-weak* or compact-to-compact topology, then $W\cap\P_d\LV_k^\infty(\R^n)$ is sequentially dense in $W$ with respect to the compact-to-weak* or compact-to-compact topology respectively. 
				\item If $W\subset \P_d\LV^0(\R^n)$ and if $W$ is closed in the compact-to-bounded topology, then $W\cap\P_d\LV_k^\infty(\R^n)$ is sequentially dense in $W$ with respect to the compact-to-bounded topology.
			\end{enumerate}
		\end{proposition}
		
		\begin{remark}\label{remark:MollifierSmoothVector}
			The proof of the previous result relies on the existence of weak integrals with respect to the three topologies. We refer to \cite{KnoerrPolynomiallocalfunctionals2025}*{Section~3.4.} for details and only note that for any $\phi\in C^\infty_c(\R^n)$ and $\Psi\in\P_d\LV(\R^n)$, the integral
			\begin{align*}
				\Psi_\phi:=\int_{\R^n}\phi(x)\pi(x)\Psi dx
			\end{align*}
			exists in the Gelfand--Pettis sense with respect to any of the three topologies in \autoref{definition:topologies} and defines an element of $\P_d\LV^\infty(\R^n)$.
		\end{remark}		
	Let us finally examine the interaction of the module structure and the three topologies. We equip $C(\R^n)$ with the topology of uniform convergence on compact subsets.
	\begin{proposition}\label{proposition:continuityModuleStructure}
		The map
		\begin{align*}
			C(\R^n)\times \P_d\LV(\R^n)&\rightarrow\P_d\LV(\R^n)\\
			(\phi,\Psi)&\mapsto \phi\bullet \Psi
		\end{align*}
		\begin{enumerate}
			\item is continuous with respect to the compact-to-bounded topology,
			\item is separately continuous with respect to the compact-to-weak* and compact-to-compact topology.
		\end{enumerate}
	\end{proposition}
	\begin{proof}
		If $B\subset C_c(\R^n)$ is bounded, then there exists a compact subset $A\subset\R^n$ such that $B\subset C_A(\R^n)$. In particular, for $\phi\in C(\R^n)$ and $\Psi\in\P_d\LV(\R^n)$, we have the estimate
		\begin{align*}
			\|\phi\bullet\Psi\|_{(B;K)}\le \sup_{x\in A}|\phi(x)| \cdot \|\Psi\|_{(B;K)},
		\end{align*}
		which shows the claim in the first case. The second case is obvious.
	\end{proof}
	Note that for $g\in\Aff(n,\R)$, we have for $\phi\in C(\R^n)$ and $\Psi\in\P_d\LV(\R^n)$,
	\begin{align}
		\label{eq:moduleEquivariant}
		\pi(g)[\phi\bullet\Psi]=[\phi\circ g^{-1}]\bullet [\pi(g)\Psi].
	\end{align}
	
	\begin{corollary}\label{corollary:smoothnessModuleStructure}
		$\P_d\LV^\infty(\R^n)$ is a module over $C^\infty(\R^n)$.
	\end{corollary}
	\begin{proof}
		Since the map 
		\begin{align*}
			C(\R^n)\times \P_d\LV(\R^n)&\rightarrow \P_d\LV(\R^n)\\
			(\phi,\Psi)&\mapsto \phi\bullet \Psi
		\end{align*}
		is bilinear, equivariant with respect to translations by Eq.~\eqref{eq:moduleEquivariant}, and continuous with respect to the compact-to-bounded topology by \autoref{proposition:continuityModuleStructure}, the claim follows directly from the fact that any element in $C^\infty(\R^n)$ is a smooth vector of the representation of the group of translations on $C(\R^n)$.
	\end{proof}
	
	\begin{corollary}\label{corollary:smoothIsAffineSmoothPreliminary}
		For $\Psi\in \P_d\LV(\R^n)^{tr}$ and $\phi\in C^\infty(\R^n)$, the map
		\begin{align*}
			\Aff(n,\R)\mapsto \P_d\LV(\R^n)\\
			g\mapsto \pi(g)[\phi\bullet \Psi]
		\end{align*}
		is smooth with respect to the compact-to-weak*, compact-to-compact, or compact-to-bounded topology.
	\end{corollary}
	\begin{proof}
		Since $\P_d\LV(\R^n)^{tr}$ is a finite dimensional representation of $\Aff(n,\R)$ by \autoref{theorem:ClassificationPolyTranslationInv}, every element in $\P_d\LV(\R^n)^{tr}$ as an $\Aff(n,\R)$-smooth vector. If $\phi\in C^\infty(\R^n)$, then the map $g\mapsto \phi\circ g$ is smooth with respect to the topology of uniform convergence on compact subsets, so the claim follows with the same argument as in the proof of \autoref{corollary:smoothnessModuleStructure}
	\end{proof}

\section{A Paley--Wiener--Schwartz-type theorem}\label{section:PWS}
	Let $0\le k\le n$ and denote by $\P_0\LV_{c,k}(\R^n)\subset \P_0\LV_k(\R^n)$ the subspace of compactly supported local functionals (i.e. with compact local support). For $\Psi\in \P_0\LV_{c,k}(\R^n)$, the support of the distribution $\widehat{\GW}(\Psi)$ is compact by definition, so its Fourier--Laplace transform is the entire function on $(\C^n)^{k+1}$ given by
	\begin{align*}
		\mathcal{F}(\widehat{\GW}(\Psi))[w_1,\dots,w_{k+1}]=\widehat{\GW}(\Psi)\left[\exp(-i\langle w_1,\cdot\rangle)\otimes\dots\otimes \exp(-i\langle w_{k+1},\cdot\rangle)\right].
	\end{align*}
	Note that this implies for $y_1,\dots,y_k\in\R^n$, $w_{k+1}\in\C^n$, 
	\begin{align}
		\notag
		&\mathcal{F}(\widehat{\GW}(\Psi))[iy_1,\dots,iy_k,w_{k+1}]\\
		\notag
		=&\int_{\R^n}\exp(-i\langle w_{k+1},\cdot\rangle)d\GW(\Psi)\left[\exp(\langle y_1,\cdot\rangle)\otimes\dots\otimes \exp(\langle y_{k},\cdot\rangle)\right]\\
		\label{eq:calculateFourierFromValuation}
		=&\frac{1}{k!}\frac{\partial^k}{\partial\lambda_1\dots\partial\lambda_k}\Big|_0\int_{\R^n}\exp(-i\langle w_{k+1},\cdot\rangle)d\Psi\left(\sum_{j=1}^k\lambda_j\exp(\langle y_j,\cdot\rangle)\right),
	\end{align}
	where we used the relation of the Goodey--Weil distributions to the polarization of $\Psi$, compare \autoref{theorem:GW}, and that $\exp(\langle y_j,\cdot\rangle)$ is a convex function for every $1\le j\le k$.\\
	
	In this section, we will examine the space of these entire functions and relate them to the modules considered in Section \ref{section:minors}. We will again identify $(\C^n)^{k+1}\cong \Mat_{n,k+1}(\C)$ with the space of complex $(n\times(k+1))$-matrices. Let us consider the space $\mathcal{O}_{\Mat_{n,k+1}(\C)}$ as a module over $\mathcal{O}_{\C^n}$ using the action of $g\in \mathcal{O}_{\C^n}$ on $F\in \mathcal{O}_{\Mat_{n,k+1}(\C)}$ given by
	\begin{align*}
		[g\odot F](w)=g\left(\sum_{j=1}^{k+1}w_j\right)F(w),\quad w\in\Mat_{n,k+1}(\C).
	\end{align*}
	Let $\widehat{\M^2_k}\subset \mathcal{O}_{\Mat_{n,k+1}(\C)}$ denote the $\mathcal{O}_{\C^n}$-submodule generated by the space $\M^2_k$ of all quadratic products of $k$-minors of the first $k$ columns of an element in $\Mat_{n,k+1}(\C)$, compare \autoref{section:minors}. For $w=(w_1,\dots,w_{k+1})\in\Mat_{n,k+1}(\C)$, we set $\mathrm{d}(w)=\sum_{j=1}^{k+1}w_j\in\C^n$.
	In this section, establish the following description of the Goodey--Weil distributions.
	\begin{theorem}
		\label{theorem:PWS_LV}
		Let $A\subset \R^n$ be compact and convex with non-empty interior and $\Psi\in \P_0\LV_k(\R^n)$ be compactly supported. Then $\Psi$ is smooth and satisfies $\locsupp\Psi\subset A$ if only if the following holds: For every $N\in\mathbb{N}$ there exists a constant $C_N>0$ such that
		\begin{align}
			\label{eq:PWSconditionValuations}
			\begin{split}
				&|\mathcal{F}(\widehat{\GW}(\Psi))[w]|\le C_N \prod_{j=1}^k|w_j|^2 \cdot (1+|\mathrm{d}(w)|)^{-N}\exp\left(h_A\left(\Im \mathrm{d}(w)\right)\right).
			\end{split}
		\end{align}
		Moreover, if an entire function belongs to $\widehat{\M^2_k}$ and satisfies estimates of the form \eqref{eq:PWSconditionValuations} for every $N\in\mathbb{N}$, then it is the Fourier--Laplace transform of the Goodey--Weil distribution of a unique smooth local functional in $\P_0\LV(\R^n)$ with local support contained in $A$.
	\end{theorem}
	
	The proof of this result requires several steps. We first establish that $\mathcal{F}(\widehat{\GW(\Psi)})\in\widehat{\M^2_k}$ in \autoref{section:FourierLaplace} using the decomposition results from \autoref{section:minors}. To that end, we examine the power series expansion of $\mathcal{F}(\widehat{\GW}(\Psi))$ in $0$ and show that the homogeneous terms belong to a corresponding module of polynomials. This step relies on a description of the restriction of  $\mathcal{F}(\widehat{\GW}(\Psi))$ to certain subspaces, which is a consequence of a restriction property of smooth local functionals in $\P_0\LV_k(\R^n)$ established in \autoref{section:restrictionSubspaces}.  Finally, we combine certain estimates for $\mathcal{F}(\widehat{\GW}(\Psi))$ for smooth $\Psi$ with the decomposition results from \autoref{section:minors} and the classical Paley--Wiener--Schwartz Theorem (see \autoref{theorem:PaleyWienerSchwartz_distributions} below) to prove \autoref{theorem:PWS_LV} in \autoref{section:ProofPWS}.

	\subsection{Restriction to subspaces}\label{section:restrictionSubspaces}
		The goal of this section is the proof of the following result.
		\begin{proposition}
			\label{proposition:RestrictionSubspaceSmoothCase}
			Let $\Psi\in \P_0\LV^\infty_k(\R^n)$ be smooth, $E\in \Gr_k(\R^n)$. Then there exists a continuous function $\Phi_E\in C(\R^n)$ such that
			\begin{align*}
				d\Psi(\pi_E^*f)=\Phi_E d(\MA_E(f)\otimes\vol_{E^\perp})
			\end{align*}
			for every $f\in \Conv(E,\R)$.
		\end{proposition}
		\begin{remark}
			The function $\Phi_E$ is in fact smooth, however, since we do not need this fact for the proof of \autoref{maintheorem:IntegralRepSmooth}, we will only show the simpler version above. Once \autoref{maintheorem:IntegralRepSmooth} is established, this more refined version follows immediately.
		\end{remark}
		Recall that we identify the space of complex Radon measures $\M(\R^n)$ with $C_c(\R^n)'$ equipped with the weak* topology. Since $\P_0\LV_0(\R^n)\cong \M(\R^n)$, we may consider $\M(\R^n)$ as a subspace of $\P_d\LV(\R^n)$. The following result may be seen as a special case of \autoref{maintheorem:IntegralRepSmooth}.
		\begin{proposition}
			\label{proposition:smoothRadonMeasure}
			Let $\nu\in \M(\R^n)$ be a measure such that the map
			\begin{align*}
				\R^n&\rightarrow\M(\R^n)\\
				x&\mapsto \pi(x)\nu
			\end{align*}
			is smooth with respect to the weak* topology. Then the following holds:
			\begin{enumerate}
				\item $\nu$ is absolutely continuous with respect to the Lebesgue measure and the density is smooth.
				\item If $\nu$ is in addition compactly supported on $[-R,R]^n$ and $\nu=f\bullet \vol$ for $f\in C^\infty_c(\R^n)$, then there exist constants $C_{N,R}>0$ independent of $\nu$ such that for all $|\alpha|\le N$,
				\begin{align*}
					\|\partial^\alpha f\|_\infty \le C_{N,R} \|(\partial_1\dots\partial_n)^{N+2}\pi(x)\nu|_{x=0}\|_{C_c(\R^n)'}.
				\end{align*}
			\end{enumerate}	
		\end{proposition}
		\begin{proof}
			Note first that the distributional derivatives of $\nu$ are given by $\partial_j \nu=-\partial_j\pi(x)\nu|_0$, since for $\phi\in C^\infty_c(\R^n)$
			\begin{align*}
				\int_{\R^n}\partial_j\phi(x)d\nu(x)=&\lim\limits_{h\rightarrow 0}\frac{1}{h}\left(
				\int_{\R^n}\phi(x+he_j)d\nu(x)-\int_{\R^n}\phi(x)d\nu(x)\right)\\
				=&\lim\limits_{h\rightarrow 0}\int_{\R^n} \phi(x)d\left(\frac{\pi(he_j)\nu-\nu}{h}\right)\\
				=&\int_{\R^n} \phi(x)d\left[\partial_j\pi(y)\nu|_{y=0}\right].
			\end{align*}
			In particular, the distributional derivatives of arbitrary order are again measures.\\
			
			Next, observe that for $\phi\in C^\infty_c(\R^n)$, the measure $\phi\bullet \nu$ has the same property, compare \autoref{corollary:smoothnessModuleStructure}. Using a smooth partition of unity, we may therefore assume that $\nu$ has compact support. \\
			Assume that $\tilde{\nu}$ is a measure with support bounded from below in all coordinate directions. Then $T(\tilde{\nu})$ given by
			\begin{align*}
				T(\tilde{\nu})[x]:=\int_{(-\infty,x_1]\times \dots(-\infty,x_n]}d\tilde{\nu}
			\end{align*}
			defines a locally bounded and measurable function on $\R^n$ with the same support property, so we may repeat the construction by identifying locally integrable functions with the corresponding absolutely continuous measure with respect to the Lebesgue measure. In particular, $T^2$ maps any compactly supported measure to a continuous function.
			Using Fubini's Theorem, we obtain for $\phi\in C^\infty_c(\R^n)$ and a compactly supported measure $\tilde{\nu}$, 
			\begin{align*}
					\int_{\R^n}\phi(x)T(\tilde{\nu})[x] dx=&\int_{\R^n} \left[\int_{[y_1,\infty)\times\dots \times [y_n,\infty)}\phi(x)dx\right]d\tilde{\nu}(y).
			\end{align*}
			If the distributional derivative $\partial_1\dots\partial_n\tilde{\nu}$ is a measure, this implies 
			\begin{align*}
				T(\partial_1\dots\partial_n\tilde{\nu})=\tilde{\nu}
			\end{align*}
			as distributions. In our case, this implies that $T^2((\partial_1\dots\partial_n)^k \nu)$ is a continuous function for every $k\in\mathbb{N}$, and so
			\begin{align*}
				\nu=T^{k}((\partial_1\dots\partial_n)^k \nu)=T^{k-2}(T^{2}((\partial_1\dots\partial_n)^k \nu))
			\end{align*}
			is of class $C^{k-2}$. This shows the first claim.\\
			
			Now assume that $\nu=f\bullet \vol$ for $f\in C^\infty_c(\R^n)$ with $\supp \nu=\supp f\subset [-R,R]^n$. Then for any multiindex $\alpha$ with $|\alpha|\le N$,
			\begin{align*}
				\partial^\alpha f=(-1)^{nN}\partial^\alpha T^{N+2}((\partial_1\dots\partial_n)^{N+2}\pi(x)\nu|_{x_0}).
			\end{align*}
			It is easy to see that $T^l((\partial_1\dots\partial_n)^N \nu)$ is supported on $[-R,R]^n$ for any $l\le N$. From the definition of $T$, we thus obtain that
			\begin{align*}
				\|\partial^\alpha f\|_\infty\le& R^{nN-|\alpha|} \|T^{2}((\partial_1\dots\partial_n)^{N+2}\pi(x)\nu|_{x=0})\|_\infty\\
				\le&R^{n(N+2)-|\alpha|} \|(\partial_1\dots\partial_n)^{N+2}\pi(x)\nu|_{x=0}\|_{C_c(\R^n)'}.
			\end{align*}
			The completes the proof of the second claim.
		\end{proof}
		
		In order to obtain \autoref{proposition:RestrictionSubspaceSmoothCase} from the previous characterization, we are going to exploit the fact that our functionals are locally determined. This step relies on the following observation, which is a simple variation of the case $k=n$ used in \cite[Lemma 4.7]{KnoerrMongeAmpereoperators2024}.
		\begin{lemma}
			\label{lemma:supportFunctionLowerDimensionalFace}
			Let $E\in\Gr_k(\R^n)$ and $P\subset E$ a polytope. For every $y_0\in \R^n\setminus E^\perp$ there exists a neighborhood $U_{y_0}$ of $y_0$ and a $(k-1)$-dimensional face $P_{y_0}$ of $P$ such that
			\begin{align*}
				h_P(y)=h_{P_{y_0}}(y)\quad\text{for all}~y\in U_{y_0}.
			\end{align*}
		\end{lemma}
		\begin{proof}
			We decompose $\R^n=E\oplus E^\perp$ and write $y=(y_1,y_2)\in E\oplus E^\perp$ for $y\in \R^n$. Fix $y_0\in\R^n\setminus E^\perp$ and write $y_0=(y_{0,1},y_{0,2})$. As $P$ is convex, the function $x\mapsto \langle y_0,x\rangle$ defined for $x\in P\subset E$ attains its maximum $h_P(y_0)$ in certain vertices of $P$. If $P$ has $m$ vertices, we may thus assume that the vertices $v_1,\dots,v_m\in E$ of $P$ satisfy
			\begin{align*}
				h_P(y_0)=&\langle y_0,v_i\rangle= \langle y_{0,1},v_i\rangle\quad \text{for }1\le i\le l,\\
				h_P(y_0)>&\langle y_0,v_i\rangle=\langle y_{0,1},v_i\rangle \quad \text{for }l+1\le i\le m,
			\end{align*}
			for some $1\le l\le m$.
			Note that this implies that $v_1,\dots,v_l$ belong to the face $P':=\{x\in P: \langle x,y_0\rangle =h_P(y_0)\}$, which is of dimension at most $k-1$ since it is the intersection of $P$ with a hyperplane in $E$ (as $y_{0,1}\ne0$), and $\dim E=k$. If $l=m$, this shows the claim, as $P'=P$ in this case. Thus assume $1\le l <m$.\\
			As $h_P$ is continuous, we may choose a neighborhood $U\subset \R^n\setminus E^\perp$ of $y_0$ such that for all $y\in U$
			\begin{align*}
				h_P(y)>\langle y,v_i\rangle  \quad \text{for }l+1\le i\le m.
			\end{align*}
			For $y\in U$, the function $x\mapsto \langle x,y\rangle$ defined on $P$ attains its maximum $h_P(y)$ on a vertex of $P$, so we find $1\le i\le l$ such that $h_P(y)=\langle y,v_i\rangle$. In particular, $h_P(y)\le h_{P'}(y)$ for all $y\in U$, which implies $h_{P'}(y)=h_P(y)$ for all $y\in U$ as $P'\subset P$.
		\end{proof}
		For the next step, we require some facts about valuations on convex bodies, and we refer to \cite{SchneiderConvexbodiesBrunn2014}*{Section~6} for a general overview. Let $\mathcal{K}(\R^n)$ denote the space of all non-empty compact and convex subsets of $\R^n$ equipped with the Hausdorff metric. Then a map $\mu:\mathcal{K}(\R^n)\rightarrow F$ into a locally convex vector space is called a valuation if
		\begin{align*}
			\mu(K)+\mu(L)=\mu(K\cup L)+\mu(K\cap L)
		\end{align*}
		for all $K,L\in\mathcal{K}(\R^n)$ such that $K\cup L$ is convex. Valuations on convex bodies are intimately related to valuations on convex functions: For $K\in\mathcal{K}(\R^n)$, its support function $h_K\in \Conv(\R ^n,\R)$ is defined by
		\begin{align*}
			h_K(y)=\sup_{x\in K}\langle x,y\rangle.
		\end{align*}
		If $\Psi:\Conv(\R^n,\R)\rightarrow F$ is a valuation, then $K\mapsto \Psi(h_K)$ defines a valuation on $\mathcal{K}(\R^n)$, compare \cite{AleskerValuationsconvexfunctions2019}*{Section~1} or \cite{KnoerrUlivelliPolynomialvaluationsconvex2026}*{Lemma 2.15}. In particular, the restriction of any continuous local functional to support functions defines a valuation on convex bodies, so we may use classification results for valuations on convex bodies to examine these functionals.\\ 
		The following was shown by Hadwiger for real-valued valuations in \cite{HadwigerVorlesungenuberInhalt1957}, however, the proof applies in the more general setting below (see also \cite{KnoerrUlivelliPolynomialvaluationsconvex2026}*{Theorem~2.13}). 
		\begin{theorem}
			\label{theorem:VolumeCharacterization}
			Let $F$ be a Hausdorff topological vector space and $\mu:\mathcal{K}(\R^n)\rightarrow F$ a continuous and translation invariant valuation that is homogeneous of degree $n$. Then there exists $v\in F$ such that $\mu(K)=\vol_n(K) v$ for all $K\in\mathcal{K}(\R^n)$.
		\end{theorem}
		
		\begin{lemma}
			\label{lemma:supportRestrictionLocalFunctional}
			Let $\Psi\in\P_0\LV_k(\R^n)$, $E\in\Gr_k(\R^n)$. For $K\in\mathcal{K}(E)$ and $x_E\in E$, we have
			\begin{align*}
				\supp\Psi(h_K(\cdot-x_E))\subset \{x_E\}\times E^\perp.
			\end{align*}
		\end{lemma}
		\begin{proof}
			Consider the map $\tilde{\Psi}:\mathcal{K}(E)\rightarrow\M(\R^n)$ given by $\tilde{\Psi}(K)=\Psi(h_K(\cdot-x_E))$. Then it is easy to see that $\tilde{\Psi}$ is a continuous and translation invariant valuation on $\mathcal{K}(E)$ that is homogeneous of degree $k=\dim E$, compare \cite{KnoerrUlivelliPolynomialvaluationsconvex2026}*{Lemma~2.15}, so by \autoref{theorem:VolumeCharacterization} there exists a measure $\nu\in\M(\R^n)$ with $\tilde{\Psi}(K)=\vol_E(K)\nu$. We have to show that $\nu[\phi]=0$ for every $\phi\in C_c(\R^n)$ with $\supp\phi\cap( \{x_E\}\times E^\perp)=\emptyset$. Let $P\subset E$ be a polytope of dimension $k$, i.e. such that $\vol_E(P)\ne 0$. By \autoref{lemma:supportFunctionLowerDimensionalFace} there exists a cover of $\R^n\setminus \{x_E\}\times E^\perp$ by open subsets $U_\alpha$ such that $h_P|_{U_\alpha}=h_{P_\alpha}|_{U_\alpha}$ for a polytope $P_{\alpha}\in \mathcal{K}(E)$ of dimension at most $(k-1)$. Choose a partition of unity $(\psi_\alpha)_\alpha$ subordinate to this cover. Let $\phi\in C_c(\R^n)$ be a function with $\supp\phi\cap( \{x_E\}\times E^\perp)=\emptyset$. Since $\Psi$ is locally determined and since $h_P$ and $h_{P_\alpha}$ coincide on a neighborhood of $\psi_\alpha\phi$ by construction, we obtain
			\begin{align*}
				\vol_E(P)\nu[\phi]=&\tilde{\Psi}(P)[\phi]=\Psi(h_P)[\phi]=\sum_\alpha \Psi(h_P)[\psi_\alpha\phi]=\sum_\alpha \Psi(h_{P_\alpha})[\psi_\alpha\phi]\\
				=&\sum_\alpha \tilde{\Psi}(P_\alpha)[\psi_\alpha\phi]=\sum_\alpha \vol_E(P_\alpha)\nu[\psi_\alpha\phi]=0,
			\end{align*}
			since $P_\alpha$ is a polytope of dimension at most $k-1$. Since $\phi$ was arbitrary and $\vol_E(P)\ne0$, this implies that $\supp\nu\subset\{x_E\}\times E^\perp$.
		\end{proof}
		
		\begin{proof}[Proof of \autoref{proposition:RestrictionSubspaceSmoothCase}]
			Using a partition of unity, we may assume that $\locsupp\Psi\subset [-R,R]^n$ for some $R>0$. Let us consider the valuation $\Psi_E\in\P_0\VConv_k(E,\M(\R^n))$ defined for $f\in\Conv(E,\R)$ by
			\begin{align*}
				\Psi_E(f):=\Psi(\pi_E^*f).
			\end{align*}
			For $x_E\in E$, we define $\mu_{E,x_E}\in \Val_k(E,\M(\R^n))$ by $\mu_{E,x_E}(K)=\Psi_E(h_K(\cdot-x_E))$. Since this is a valuation of degree $k$ on a $k$-dimensional space, \autoref{theorem:VolumeCharacterization} implies that there exists $\nu_{E,x_E}\in \M(\R^n)$ such that
			\begin{align*}
				\Psi_E(h_K(\cdot-x_E))=\mu_{E,x_E}(K)=\vol_k(K) \nu_{E,x_E}.
			\end{align*}
			\autoref{lemma:supportRestrictionLocalFunctional} implies that $\nu_{E,x_E}$ is supported on $\{x_E\}\times E^\perp$. We may thus consider $\nu_{E,x_E}$ as an element of $\M(E^\perp)$. Moreover, since $\locsupp\Psi\subset [-R,R]^n$ by assumption, $\supp \nu_{E,x_E}\subset \{x_E\}\times [-R,R]^{n-k}$.\\
			
			Since the function $h_K(\cdot-x_E)$ is invariant under translations in $E^\perp$, we obtain for $x_{E^\perp}\in E^\perp$
			\begin{align*}
				\pi(x_{E^\perp})\nu_{E,x_E}=\frac{1}{\vol_E(B_E(0))}[\pi(x_{E^\perp})\Psi_E](h_{B_E(0)}(\cdot-x_E)),
			\end{align*}
			where $B_E(0)\subset E$ denotes the unit ball.	In particular, the map 
			\begin{align*}
				E^\perp&\rightarrow\M(E^\perp)\\
				x_{E^\perp}&\mapsto \pi(x_{E^\perp})\nu_{E,x_E}
			\end{align*}
			is smooth with respect to the weak* topology on $\M(E^\perp)$. From \autoref{proposition:smoothRadonMeasure}, we obtain a smooth function $\Phi_{x_E}\in C^\infty(E^\perp)$ with $\supp \Phi_{x_E}\subset [-R,R]^{\dim E^\perp}$ such that
			\begin{align*}
				d\nu_{E,x_E}=\Phi_{x_E}d[\delta_{x_E}\otimes \vol_{E^\perp}].
			\end{align*}
			Set $\Phi_E(x_E,x_{E^\perp}):=\Phi_{x_E}(x_{E^\perp})$. From \autoref{proposition:smoothRadonMeasure} we obtain for every $L>0$ a constant $C(L)>0$ independent of $x_E,x_E'\in E$ with $|x_E|,|x_E'|\le L$ such that
			\begin{align*}
				&\|\Phi_E(x_E,\cdot)-\Phi_E(x_E',\cdot)\|_\infty\\
				&\le C(L)\|\left((\partial_1\dots\partial_n)^{2}[\pi(x)\Psi]|_{-x_E}-(\partial_1\dots\partial_n)^{2}[\pi(x)\Psi]|_{-x_E'}\right)(h_{B_E(0)})\|_{C_c(\R^n)'}.
			\end{align*}
			Since $x\mapsto \pi(x)\Psi$ is smooth in the compact-to-bounded topology, the right hand side of this inequality converges to zero for $x_E'\rightarrow x_E$, and it is easy to see that this implies that $\Phi_E$ is continuous. \\
			
			Now define $\tilde{\Psi}_E\in \VConv_k(E,\M(\R^n))$ by
			\begin{align*}
				\tilde{\Psi}(f_E;B)=\int_{B}\Phi_E(x)d(\MA_E(f_E)\otimes \vol_{E^\perp}) 
			\end{align*}
			for $f_E\in\Conv(E,\R)$, $B\subset\R^n$ bounded Borel set. Then $\tilde{\Psi}_E(h_K(\cdot-x_E))=\Psi_E(h_K(\cdot-x_E))$ for every $x_E\in E$. \autoref{proposition:RestrictionHomValuationLowerDimensionalSupportFunctions} thus implies $\tilde{\Psi}_E=\Psi_E$, which completes the proof.
		\end{proof}
		
		Let us add the following interesting consequence of \autoref{lemma:supportRestrictionLocalFunctional} for the continuity properties of elements in $\P_0\LV(\R^n)$. 
		\begin{corollary}
			\label{corollary:DuallyEpiNotContinuousWRTStrongTopology}
			Let $1\le k\le n$ and $\Psi\in\P_0\LV_k(\R^n)$. If $\Psi:\Conv(\R^n,\R)\rightarrow\M(\R^n)$ is continuous with respect to the strong topology, then $\Psi=0$.
		\end{corollary}
		\begin{proof}
			Assume that $\Psi\ne 0$. Combining \autoref{proposition:RestrictionHomValuationLowerDimensionalSupportFunctions} with \autoref{lemma:supportRestrictionLocalFunctional}, we obtain a $k$-dimensional subspace $E\in\Gr_k(\R^n)$, $K\in\mathcal{K}(E)$, and $x_E\in E$ such that $\Psi(h_K(\cdot-x_E))\ne 0$ and $\locsupp\Psi(h_K(\cdot-x_E))\subset \{x_E\}\times E^\perp$. In particular, there exists $\phi\in C_c(\R^n)$ with $\|\phi\|_\infty\le 1$ and  $\Psi(h_K(\cdot-x_E))[\phi]>0$. Let $(x_j)_j$ be a sequence in $E$ converging to $x_E$ with $x_j\ne x_E$ for every $j\in \mathbb{N}$. By multiplying $\phi$ with continuous cut-off functions, we obtain functions $\phi_j$ with $\supp\phi_j\cap \{x_j\}\times E^\perp=\emptyset$ for every $j\in\mathbb{N}$ and $\phi_j=\phi$ on $\{x_E\}\times E^\perp$. Then for every $j\in\mathbb{N}$,
			\begin{align*}
				\Psi(h_K(\cdot-x_E))[\phi_j]-\Psi(h_K(\cdot-x_j))[\phi_j]=\Psi(h_K(\cdot-x_E))[\phi],
			\end{align*}
			and therefore 
			\begin{align*}
				\|\Psi(h_K(\cdot-x_E))-\Psi(h_K(\cdot-x_j))\|_{C_{\supp\phi}(\R^n)'}\ge|\Psi(h_K(\cdot-x_E))[\phi]|>0.
			\end{align*}
			However, $h_{K}(\cdot-x_j)$ converges to $h_K(\cdot-x_E)$ for $j\rightarrow\infty$, which shows that $\Psi:\Conv(\R^n,\R)\rightarrow\M(\R^n)$ is not continuous with respect to the strong topology on $\M(\R^n)$. 
		\end{proof}
		\begin{remark}
			For $d>0$, this result no longer holds: The functional $\Psi\in\P_d\LV_l(\R^n)$, $d\ge l$, given by  $\Psi(f)[\phi]:=\int_{\R^n}\phi(x)f(x)^l dx$ is continuous with respect to the strong topology.
		\end{remark}
	
	\subsection{The Fourier--Laplace transform of Goodey--Weil distributions}\label{section:FourierLaplace}
		
		\begin{lemma}
			\label{lemma:continuityFourier_F}
			Let $A\subset \R^n$ be compact and convex with non-empty interior. There exists a constant $C(A)>0$ such that for every $\Psi\in   \P_0\LV_{c,k}(\R^n)$ with $\locsupp\Psi\subset A$,
			\begin{align*}
				|\mathcal{F}(\widehat{\GW}(\Psi))[w_1,\dots,w_{k+1}]|\le C(A)\|\Psi\|_{A,1} \left(\prod_{j=1}^{k}(1+|w_j|)\right)^3 e^{\sum_{j=1}^{k+1}h_A(\Im(w_j))}.
			\end{align*}
		\end{lemma}
		\begin{proof}
			Let $\mu_{w_{k+1}}:=\Psi[\exp(-i\langle w_{k+1},\cdot\rangle)]\in \P_0\VConv_{k}(\R^n,\C)$. Then we have $\supp\mu_{w_{k+1}}\subset A$ by \autoref{proposition:characterizationSupport}. By \cite[Proposition 4.3]{KnoerrPaleyWienerSchwartz2025} there exists a constant $C(A)>0$ such that for every $\delta\in(0,1)$,
			\begin{align*}
				&|\mathcal{F}(\GW(\mu_{w_{k+1}})[w_1,\dots,w_k]|\\
				&\le C(A)\|\mu_{w_{k+1}}\|_{A,\delta}\prod_{j=1}^k(1+|w_j|)^2 e^{\sum_{j=1}^kh_{A+\delta B_{1}(0)}(\Im(w_j))},
			\end{align*}
			where $\|\mu_{w_{k+1}}\|_{A,\delta}=\sup\left\{|\mu_{w_{k+1}}(f)|: f\in\Conv(\R ^n,\R), \sup_{x\in A+\delta B_1(0)}|f(x)|\le 1 \right\}$. Since $\Psi$ is supported on $A$, we obtain
			\begin{align*}
				&\|\mu_{w_{k+1}}\|_{A,\delta}=\sup\left\{|\mu_{w_{k+1}}(f)|:f\in\Conv(\R^n,\R),~\sup_{x\in A+\delta B_{1}(0)}|f(x)|\le 1\right\}\\
				&\quad=\sup\left\{|\Psi(f)[e^{-i\langle w_{k+1},\cdot\rangle}]|:f\in\Conv(\R^n,\R),~\sup_{x\in A+\delta B_{1}(0)}|f(x)|\le 1\right\}\\
				&\quad\le \|\Psi\|_{A,\delta} e^{h_{A}(\Im(w_{k+1}))}.
			\end{align*}
			Combining these estimates with Eq.~\eqref{eq:relationLocalSeminorms} from \autoref{proposition:semiNormslocalFunctional}, we obtain a constant $\tilde{C}(A)>0$ such that for $\delta\in(0,1]$,
			\begin{align*}
				&|\mathcal{F}(\widehat{\GW}(\mu))[w]|=|\mathcal{F}(\GW(\mu_{w_{k+1}})[w_1,\dots,w_k]|\\
				\le&\frac{\tilde{C}(A)}{\delta^k}\|\Psi\|_{A,1}\prod_{j=1}^k(1+|w_j|)^2 e^{h_A(\Im(w_{k+1}))+\sum_{j=1}^{k}h_{A+\delta B_{1}(0)}(\Im(w_j))}.
			\end{align*}
			For every fixed $w\in \Mat_{n,k+1}(\C)$, the choice of $\delta:=\max\left(\frac{1}{1+|w_j|}: 1\le j\le k\right)$ (which implies $\delta\in(0,1]$) provides the estimate
			\begin{align*}
				&|\mathcal{F}(\widehat{\GW}(\mu))[w]|\le\tilde{C}(A)\|\Psi\|_{A,1}\prod_{j=1}^{k}(1+|w_j|)^3 e^{k+\sum_{j=1}^{k+1}h_{A}(\Im(w_j))},
			\end{align*}
			which completes the proof.
		\end{proof}
	
		In order to apply the result from Section \ref{section:minors}, we will need to change coordinates. For $\Psi\in \P_0\LV_{c,k}(\R^n)$ we denote by $\F(\Psi)\in \mathcal{O}_{\Mat_{n,k+1}(\C)}$ the entire function on $\Mat_{n,k+1}(\C)$ defined by
		\begin{align}\label{eq:DefinitionF}
			\F(\Psi)[w]=\frac{k!}{(-1)^k}\mathcal{F}(\widehat{\GW}(\Psi))\left[w_1,\dots,w_k,w_{k+1}-\sum_{j=1}^kw_j\right].
		\end{align}
		Note that
		\begin{align}\label{eq:relationF_FourierGW}
			\mathcal{F}(\widehat{\GW}(\Psi))[w]=\frac{(-1)^k}{k!}\F(\Psi)\left[w_1,\dots,w_k,\sum_{j=1}^{k+1}w_j\right].
		\end{align}
		\begin{corollary}
			\label{corollary:EstimateF}
			Let $A\subset \R^n$ be compact and convex with non-empty interior. There exists a constant $C(A)>0$ such that for every $\Psi\in  \P_0\LV_{c,k}(\R^n)$ with $\locsupp\Psi\subset A$,
			\begin{align*}
				&|\F(\Psi)[w]|\\
				&\le C(A)\|\Psi\|_{A,1} \left(\prod_{j=1}^{k}(1+|w_j|)\right)^3 e^{h_A(\Im(w_{k+1})+\sum_{j=1}^kh_A(\Im(w_j))+h_A(-\Im(w_j))}.
			\end{align*}
				In particular, the restriction of the map
			$\F:  \P_0\LV_{c,k}(\R^n)\rightarrow \mathcal{O}_{\Mat_{n,k+1}(\C)}$ to local functionals with support contained in $A$ is continuous in the compact-to-bounded topology.
		\end{corollary}
		\begin{proof}
			This follows directly from \autoref{lemma:continuityFourier_F} and the estimate
			\begin{align*}
				h_A\left(\Im(w_{k+1})-\sum_{j=1}^k\Im(w_j)\right)\le h_A\left(\Im(w_{k+1})\right)+\sum_{j=1}^kh_A\left(-\Im(w_j)\right),
			\end{align*}
			which follows from the subadditivity of the support function.
		\end{proof}
		Let us consider $\mathcal{O}_{\Mat_{n,k+1}(\C)}$ as a module over $\mathcal{O}_{\C^n}$ by identifying $\mathcal{O}_{\C^n}$ with functions that depend on the variable $w_{k+1}$ only. Note that this corresponds to the $\O_{\C^n}$-module structure considered in \autoref{section:minors} under the coordinate change in Eq.~\eqref{eq:DefinitionF}.	We will show that the functions $\F(\Psi)$, $\Psi\in\P_0\LV_{c,k}(\R^n)$ belong to the $\mathcal{O}_{\C^n}$-submodule generated by $\M^2_k\subset \mathcal{O}_{\Mat_{n,k+1}(\C)}$. We start by calculating this function for local functionals constructed from translation invariant elements in $\P_0\LV_k(\R^n)$. Our argument requires two observations from \cite{KnoerrMongeAmpereoperators2024}, which concern locally determined, continuous valuations $\Psi:\Conv(\R^n,\R)\rightarrow\M(\R^n)$ that are invariant with respect to translations. Due to \autoref{theorem:LocalFuncValuations}, this space (denoted by $\MAVal(\R^n)$ in \cite{KnoerrMongeAmpereoperators2024}) coincides with $\P_0\LV(\R ^n)^{tr}$, and so we will use the notation $\MAVal_k(\R^n):=\P_0\LV_k(\R^n)^{tr}$ in order to be consistent with the notation in \cite{KnoerrMongeAmpereoperators2024}.
		\begin{theorem}[\cite{KnoerrMongeAmpereoperators2024}*{Theorem 6.16 and Corollary 6.17}]
			\label{theorem:FourierMA}
			For every $\Psi\in \MAVal_k(\R^n)$ there exists a unique $Q(\Psi)\in \M_k^2$ such that for all $\phi\in C_c(\R^n)$ we have
			\begin{align*}
				\mathcal{F}(\GW(\Psi[\phi]))[w_1,\dots,w_k]=\frac{(-1)^k}{k!}Q(\Psi)[w_1,\dots,w_k] \mathcal{F}(\phi)\left[\sum_{j=1}^k w_j\right].
			\end{align*}
		\end{theorem}
		\begin{lemma}[\cite{KnoerrMongeAmpereoperators2024}*{Corollary 6.19}]
			\label{lemma:QBijectiveMA}
			The map
			\begin{align*}
				\MAVal_k(\R^n)&\rightarrow \M^2_k\\
				\Psi&\mapsto Q(\Psi)
			\end{align*}
			is bijective.
		\end{lemma}
		This implies the following for the $C_c(\R^n)$-submodule of $\P_0\LV_{c,k}(\R^n)$ generated by $\MAVal_k(\R^n)$.
		\begin{lemma}
			\label{lemma:FourierMA}
			For $\Psi\in \MAVal_k(\R^n)$ and $\phi\in C_c(\R^n)$, we have
			\begin{align*}
				\mathcal{F}(\widehat{\GW}(\phi\bullet \Psi))[w]=\frac{(-1)^k}{k!}Q(\Psi)[w_1,\dots,w_k]\mathcal{F}(\phi)\left[\sum_{j=1}^{k+1}w_j\right].
			\end{align*}
			In particular, $\F(\phi\bullet \Psi)[w]=Q(\Psi)[w_1,\dots,w_k]\mathcal{F}(\phi)[w_{k+1}]$.
		\end{lemma}
		\begin{proof}
			Using \autoref{theorem:FourierMA}, the definition of $\mathcal{F}(\widehat{\GW}(\phi\bullet \Psi))$ implies
			\begin{align*}
				&\mathcal{F}(\widehat{\GW}(\phi\bullet \Psi))[w]=\mathcal{F}(\GW(\Psi[\phi\exp(-i\langle w_{k+1},\cdot\rangle)]))[w_1,\dots,w_k]\\
				=&\frac{(-1)^k}{k!}Q(\Psi)[w_1,\dots,w_k] \mathcal{F}(\phi\exp(-i\langle w_{k+1},\cdot\rangle)]))\left[\sum_{j=1}^k w_j\right]\\
				=&\frac{(-1)^k}{k!}Q(\Psi)[w_1,\dots,w_k]\mathcal{F}(\phi)\left[\sum_{j=1}^{k+1}w_j\right].
			\end{align*}
			The second claim follows directly from the definition of $\F(\phi\bullet\Psi)$.
		\end{proof}
		
		Next, we investigate the homogeneous terms of the power series expansion of these functions. Let $\M\subset \Poly(\Mat_{n,k+1}(\C))$ denote the subspace generated by the lowest order terms of $\F(\Psi)$ for $\Psi\in \P_0\LV_{c,k}(\R^n)$. As in Section \ref{section:minors}, we consider $\Poly(\Mat_{n,k+1}(\C))$ as a module over $\Poly(\C^n)$ by identifying $\Poly(\C^n)$ with the subspace of polynomials in the variable $w_{k+1}$.
		\begin{lemma}
			\label{lemma:modulePropertiesLowestOrderTerms}
			$\M$ is spanned by the lowest order terms of the power series expansion of $\F(\Psi)$ for smooth local functionals $\Psi\in\P_0\LV_{c,k}(\R^n)$. Moreover, $\M$ is a $\Poly(\C^n)$-submodule of $\Poly(\Mat_{n,k+1}(\C))$.
		\end{lemma}
		\begin{proof}
			Let us show that every lowest order term of $\F(\widehat{\GW}(\Psi))$ for $\Psi\in\P_0\LV_{c,k}(\R^n)$ is the lowest order term of a smooth local functional. Let $\Psi\in\P_0\LV_{c,k}(\R^n)$ and $\phi\in C^\infty_c(\R^n)$ be a function with $\int_{\R^n}\phi(x)dx=1$. Then the weak integral (with respect to the compact-to-weak* topology)
			\begin{align*}
				\Psi_\phi=\int_{\R^n}\phi(x)\pi(x)\Psi dx
			\end{align*}
			exists and defines a smooth local functional, compare \autoref{remark:MollifierSmoothVector}. From the definition of $\mathcal{F}\circ\widehat{\GW}$, we obtain
			\begin{align*}
				&\mathcal{F}(\widehat{\GW}(\Psi_\phi))[w]\\
				=&\int_{\R^n}\phi(x)\widehat{\GW}(\Psi)\left[\exp(-i\langle w_1,\cdot+x\rangle)\otimes\dots\otimes \exp(-i\langle w_{k+1},\cdot+x\rangle)\right]dx\\
				=&\int_{\R^n}\phi(x)\exp\left(-i\sum_{j=1}^{k+1}\langle w_j,x\rangle\right)dx\cdot \mathcal{F}(\widehat{\GW}(\Psi))[w],
			\end{align*}
			and therefore
			\begin{align*}
				\F(\Psi_\phi)[w]=\mathcal{F}(\phi)[w_{k+1}]\F(\Psi)[w].
			\end{align*}
			Since $\mathcal{F}(\phi)[0]=\int_{\R^n}\phi(x)dx=1$, the lowest order term of the power series expansion of this function coincides with the lowest order term of $\F(\Psi)$, which shows the claim.\\
			
			For the second claim, observe that it is now sufficient to show that the product of the lowest order term of  $\F(\Psi_\phi)[w]$ and $w_{j,k+1}$, $1\le j\le n$, belongs to $\M$. However, this product is the lowest order term of
			\begin{align*}
				w_{j,k+1}\F(\Psi_\phi)[w]=\mathcal{F}((-i\partial_j)\phi)[w_{k+1}]\F(\Psi)[w]=\F(\Psi_{-i\partial_j\phi})[w],
			\end{align*}
			which belongs to $\M$.
		\end{proof}
		
		Recall that $\widetilde{\M^2_k}$ denotes the submodule of $\Poly(\Mat_{n,k+1}(\C))$ generated by quadratic products of the $k$-minors involving the first $k$ columns.	
		\begin{lemma}
			\label{lemma:SubmoduleContainedMinors}
			$\M\subset \widetilde{\M^2_k}$.
		\end{lemma}
		\begin{proof}
		Note that \autoref{lemma:modulePropertiesLowestOrderTerms} implies that it is sufficient to show that the lowest order term of the power series expansion of $\F(\Psi)$  belongs to $\widetilde{\M^2_k}$ for every smooth local functional $\Psi\in \P_0\LV_{c,k}(\R^n)$. Let $E\in\Gr_k(\R^n)$. We can apply \autoref{proposition:RestrictionSubspaceSmoothCase} to obtain a continuous function $\Phi_E\in C(\R^n)$ such that
			\begin{align}
				\label{eq:formulaRestriction}
				d\Psi(\pi_E^*f)=\Phi_Ed(\MA_E(f)\otimes\vol_{E^\perp})
			\end{align}
			for all $f\in \Conv(E,\R)$. Note that $\Phi_E$ has compact support since $\Psi$ is compactly supported.	For $y_1,\dots,y_k\in E$ and $w_{k+1}\in\C^n$, we may combine Eq.~\eqref{eq:calculateFourierFromValuation} and Eq.~\eqref{eq:formulaRestriction} to obtain
			\begin{align*}
				&\mathcal{F}(\widehat{\GW}(\Psi))[iy_1,\dots,iy_k,w_{k+1}]\\
				=&\frac{1}{k!}\frac{\partial^k}{\partial \lambda_1\dots\partial\lambda_k}\Big|_0\int_{\R^n}\exp(-i\langle w_{k+1},x\rangle)\Phi_E(x)d\left[\MA_E\left(\sum_{j=1}^k\lambda_j\exp(\langle y_j,\cdot\rangle);\cdot\right)\otimes \vol_{E^\perp}\right],
			\end{align*}
			which implies
			\begin{align*}
				&\mathcal{F}(\widehat{\GW}(\Psi))[iy_1,\dots,iy_k,w_{k+1}]
				=&\frac{\det\nolimits_k\left(\langle y_i,y_j\rangle\right)_{i,j=1}^k}{k!} \mathcal{F}(\Phi_E)\left[w_k+\sum_{j=1}^k iy_j\right].
			\end{align*}
			Thus for $w_1,\dots,w_k\in E\otimes\C$, $w_{k+1}\in\C^n$,
			\begin{align*}
				\mathcal{F}(\widehat{\GW}(\Psi))[w]=\frac{(-1)^k}{k!}\det\nolimits_k\left(\langle w_i,w_j\rangle\right)_{i,j=1}^k \mathcal{F}(\Phi_E)\left[\sum_{j=1}^{k+1} w_j\right],
			\end{align*}
			since both sides are holomorphic functions on $(E\otimes\C)^k\times \C^n$ that coincide on $(iE)^k\times \C^n$. From the definition of $\F$, we thus obtain for $w_1,\dots,w_k\in E\otimes \C$ and $w_{k+1}\in \C^n$
			\begin{align*}
				\mathcal
				\F(\Psi)[w]=\det\nolimits_k\left(\langle w_i,w_j\rangle\right)_{i,j=1}^k \mathcal{F}(\Phi_E)\left[w_{k+1}\right].
			\end{align*}
			Since this holds for every subspace $E\in\Gr_k(\R^n)$, the lowest order term $P$ of the power series expansion of $\F(\Psi)$ satisfies
			\begin{align*}
				P|_{(E\otimes \C)^k\times \C^n}(w_1,\dots,w_{k+1})=\det\nolimits_k\left(\langle w_i,w_j\rangle\right)_{i,j=1}^kQ_E(w_{k+1})
			\end{align*}
			for $w_1,\dots,w_k\in E\otimes \C$, $w_{k+1}\in\C^n$, for a polynomial $Q_E\in\Poly(\C^n)$. Since $(w_1,\dots,w_k)\mapsto \det\nolimits_k\left(\langle w_i,w_j\rangle\right)_{i,j=1}^k$ is a multiple of the square of the determinant function on $E$, it is now easy to see that $P$ satisfies the condition in \autoref{lemma:PolynomialFromComplexifiedSpaces}, so \autoref{proposition:CharacterizationModuleSquareMinors} implies $P\in\widetilde{\M^2_k}$. Thus $\M\subset \widetilde{\M^2_k}$.
		\end{proof}
		
		Next we are going to show that every term in the power series expansion of $\F(\Psi)$, $\Psi\in\P_0\LV_{c,k}(\R^n)$, belongs to $\widetilde{\M^2_k}$ in order to apply \autoref{theorem:GroebnerAlgModuleMinor}.
		
		\begin{lemma}
			\label{lemma:FbelongsToModule}
			\begin{enumerate}
				\item Let $N\in\mathbb{N}$. For every element $Q\in\widetilde{\M^2_k}$ of degree at most $N+2k$, there exists $\Psi\in\P_0\LV^\infty_{c,k}(\R^n)$ such that the power series expansion of $\F(\Psi)$ coincides with $Q$ up to order $N+2k$. 
				\item Every homogeneous term of the power series expansion of $\F(\Psi)$ for $\Psi\in\P_0\LV_{c,k}(\R^n)$ belongs to $\widetilde{\M^2_k}$.
			\end{enumerate}
		\end{lemma}
		\begin{proof}
			Note that the second claim follows from the first: Assume that $\Psi\in\P_0\LV_{c,k}(\R^n)$. If there is a homogeneous term in the power series expansion of $\F(\Psi)$ that does not belong to $\widetilde{\M^2_k}$, then there exists a maximal $N\in\mathbb{N}$ such that the power series expansion up to order $2k+N$ belongs to $\widetilde{\M^2_k}$ but the term of order $2k+N+1$ does not. Let $Q$ be the power series expansion of $\F(\Psi)$ up to order $2k+N$, and choose $\tilde{\Psi}\in\P_0\LV_{c,k}(\R^n)$ such that $\F(\tilde{\Psi})$ coincides with $Q$ up to order $2k+N+1$. Then $\Psi-\tilde{\Psi}_0$ has the property that the lowest order term of $\F(\Psi-\tilde{\Psi}_0)$ is given by the homogeneous term of degree $2k+N+1$ of the power series expansion of $\F(\Psi)$, which does not belong to $\widetilde{\M^2_k}$ by assumption. However, this contradicts \autoref{lemma:SubmoduleContainedMinors}. Thus every homogeneous term in the power series expansion belongs to $\widetilde{\M^2_k}$.\\
			
			In order to see that the first claim holds, write $Q=\sum_{j=1}^{N_{n,k}}P_j(w_{k+1})Q(\Psi_j)$ for a basis $\Psi_j$, $1\le j\le N_{n,k}$, of $\MAVal_k(\R^n)$, which is possible due to \autoref{lemma:QBijectiveMA}. We may assume that $P_j$ is a polynomial of degree at most $N$. Choose functions $\phi_j\in C_c^\infty(\R^n)$ such that the power series expansion of $\mathcal{F}(\phi_j)$ up to order $N$ coincides with $P_j$ (see for example \cite[Proposition 4.11]{KnoerrPaleyWienerSchwartz2025}). Then the local functional $\tilde{\Psi}:=\sum_{j=1}^{N_{n,k}}\phi_j\bullet \Psi_j\in\P_0\LV_{c,k}^\infty(\R^n)$ has the desired property due to \autoref{lemma:FourierMA}.
		\end{proof}
		The following result provides a general description of the Fourier--Laplace transform of $\widehat{\GW}(\Psi)$ for $\Psi\in\P_0\LV_{c,k}(\R^n)$.
		\begin{proposition}
			For every compact and convex set $A\subset \R^n$ with non-empty interior and every basis $\Psi_j$ of $\MAVal_k(\R^n)$, there exists a constant $C(A)>0$ such that for every $\Psi\in   \P_0\LV_{c,k}(\R ^n)$ with $\locsupp\Psi\subset A$ there exist unique functions $g\in \mathcal{O}_{\C^n}$ with
			\begin{align}
				\label{eq:presentationPsi}
				\mathcal{F}(\widehat{\GW}(\Psi))[w]=\sum_{j=1} g_j\left(\sum_{j=1}^{k+1}w_j\right)Q(\Psi_j)[w_1,\dots,w_k]
			\end{align}
			where 
			\begin{align*}
				|g_j(z)|\le& C(A)\|\Psi\|_{A;1}(1+|z|)^{nN_{n,k}} \exp(h_A(\Im z)).
			\end{align*}
		\end{proposition}
		\begin{proof}
			By \autoref{lemma:FbelongsToModule}, every homogeneous term in the power series expansion of $\F(\Psi)$ belongs to the $\Poly(\C^n)$-module generated by $\widetilde{\M}^2_k$. We may therefore apply \autoref{theorem:GroebnerAlgModuleMinor} to obtain functions $g_j\in\mathcal{O}_{\C^n}$ such that
			\begin{align*}
				\F(\Psi)[w]=\sum_{j=1}^{N_{n,k}}g_j(w_{k+1})Q(\Psi_j)[w_1,\dots,w_k],
			\end{align*}
			where
			\begin{align}
				\label{eq:estimateG}
				|g_j(z)|\le  C\left(1+|z|\right)^{nN_{n,k}}\sup_{\zeta\in D_1(0,\dots,0,z)}|\F(\Psi)[\zeta]|,
			\end{align}
			so the claim follows from the estimate in \autoref{lemma:continuityFourier_F} and the relation between $\F(\Psi)$ and $\mathcal{F}(\widehat{\GW}(\Psi))$.
		\end{proof}
		We obtain the following two corollaries.	Recall that $\mathrm{d}(w)=\sum_{j=1}^{k+1}w_j$ for $w\in\Mat_{n,k+1}(\C)$.
		\begin{corollary}\label{corollary:GeneralEstimateFourierInTermsOfF}
			There is a constant $C>0$ such that for all $\Psi\in \P_0\LV_{c,k}(\R^n)$,
			\begin{align*}
				|\mathcal{F}(\widehat{\GW}(\Psi))[w]|\le C(1+|\mathrm{d}(w)|)^{nN_{n,k}}\prod_{j=1}^k|w_j|^2\sup_{\zeta\in D_1(0,\dots,0,d(w))}|\F(\Psi)[\zeta]|.
			\end{align*}
		\end{corollary}
		\begin{proof}
			This follows by individually estimating the terms in the sum in Eq.~\eqref{eq:presentationPsi} using Eq.~\eqref{eq:estimateG} and the fact that $Q(w_1,\dots,w_k)$ is homogeneous of degree $2$ in each argument and therefore bounded by a multiple of $\prod_{j=1}^k|w_j|^2$.
		\end{proof}
		\begin{corollary}
			For every $\Psi\in  \P_0\LV_{c,k}(\R^n)$,  $\mathcal{F}(\widehat{\GW}(\Psi))\in\widehat{\M^2_k}$.
		\end{corollary}
	
	\subsection{Proof of \autoref{theorem:PWS_LV}}
		\label{section:ProofPWS}
		The proof of the Paley--Wiener--Schwartz-type characterization of smooth local functionals with compact local support in \autoref{theorem:PWS_LV} is based on the following version of the classical Paley--Wiener--Schwartz Theorem for smooth functions with compact support.		
		\begin{theorem}[\cite{Hoermanderanalysislinearpartial2003}*{Theorem 7.3.1}]
			\label{theorem:PaleyWienerSchwartz_distributions}
			Let $A\subset \R^n$ be compact and convex with non-empty interior. If $F$ is an entire function on $\C^n$ such that for every $N\in\mathbb{N}$ there exists a constant $C_N>0$ with
			\begin{align*}
				|F(z)|\le C_N (1+|z|)^{-N}e^{h_A(\Im z)} \quad\text{for}~z\in\C^n,
			\end{align*} 
			then $F$ is the Fourier--Laplace transform of a function in $C^\infty_c(\R^n)$ with support contained in $A$. Conversely, the Fourier--Laplace transform of any smooth function with support contained in $A$ satisfies an estimate of this form for every $N\in\mathbb{N}$.
		\end{theorem}
		The key estimate is contained in the following lemma. 
		\begin{lemma}
			\label{lemma:PWS_SatisfiedBySmoothVLV}
			Let $A\subset \R^n$ be compact and convex with non-empty interior and let $\Psi\in\P_0\LV_k(\R^n)$ be smooth with $\locsupp\Psi\subset A$. For every $N\in\mathbb{N}$ there exists a constant $C_N>0$ such that 
			\begin{align*}
				&|\F(\Psi)[w_1,\dots,w_{k+1}]|\\
				\le& C_N\prod_{j=1}^{k}(1+|w_j|)^3e^{\sum_{j=1}^k h_A(\Im(w_{j}))+h_A(-\Im(w_{j}))}(1+|w_{k+1}|)^{-N}e^{h_A(\Im(w_{k+1}))}
			\end{align*}
		\end{lemma}
		\begin{proof}
			The proof of this result is nearly identical to \cite[Lemma 5.1]{KnoerrPaleyWienerSchwartz2025}. For completeness we include the argument with the necessary modifications. For $w\in\Mat_{n,k+1}(\C)$ and $x\in\R^n$,
			\begin{align*}
				\mathcal{F}(\widehat{\GW}(\pi(x)\Psi))[w]=&\widehat{\GW}(\Psi)[\exp(-i\langle w_1,\cdot+x\rangle)\otimes\dots\otimes\exp(-i\langle w_{k+1},\cdot+x\rangle)]\\
				=&\exp\left(-i\left\langle\sum_{j=1}^{k+1}w_j,x\right\rangle \right)\mathcal{F}(\widehat{\GW}(\Psi))[w].
			\end{align*}	
			In particular,
			\begin{align*}
				\exp\left(-i\left\langle w_{k+1},x\right\rangle \right)\F(\Psi)[w]=\F(\pi(x)\Psi)[w].
			\end{align*}
			Since the restriction of $\F$ to local functionals with bound on their support is continuous with respect to the compact-to-bounded topology by \autoref{corollary:EstimateF}, the right hand side defines a smooth function in $x$ for every $w\in \Mat_{n,k+1}(\C)$ and $\partial_x^\alpha\F(\pi(x)\Psi)[w]=\F(\partial_x^\alpha\pi(x)\Psi)[w]$ for any multiindex $\alpha\in \mathbb{N}^n$. Thus
			\begin{align*}
				(-i)^{|\alpha|}w_{k+1}^\alpha\exp\left(i\left\langle w_k,x\right\rangle \right)\F(\Psi)[w]=&\F(\partial_{x}^{\alpha}\pi(x)\Psi)[w].
			\end{align*}
			If we evaluate this expression in $x=0$ and sum over the appropriate indices $\alpha$ with $|\alpha|= N$, we obtain 
			\begin{align*}
				|w_{k+1}|^{N}\cdot |\F(\Psi)[w]|\le&|w_{k+1}|_1^{N}\cdot |\F(\Psi)[w]\le  n^{N}\max_{|\alpha|= N}\left|\F(\partial_x^\alpha\pi(x)\Psi|_0)\left[w\right]\right|,
			\end{align*}
			which implies
			\begin{align*}
				(1+|w_{k+1}|)^N|\F(\Psi)[w]|
				\le& 2^{N}n^{N}\max_{|\alpha|\le N}\left|\F(\partial_{x}^{\alpha}\pi(x)\Psi|_0)\left[w\right]\right|.
			\end{align*}
			Using \autoref{corollary:EstimateF}, we obtain a constant $C(A,k)$ depending on $k$ and $A\subset\R^n$ only such that for $N\in\mathbb{N}$,
			\begin{align*}
				&|\F(\Psi)[w]|\le (1+|w_{k+1}|)^{-N}(2n)^{N}\max_{|\alpha|\le N}|\F(\partial_{x}^{\alpha}\pi(x)\Psi|_0)\left[w\right]|\\
				\le&(1+|w_{k+1}|)^{-N} (2n)^{N}C(A,k)\prod_{j=1}^{k}(1+|w_j|)^3  \cdot \max_{|\alpha|\le N}\|\partial_{x}^{\alpha}\pi(x)\Psi|_0\|_{A,1}\\
				&e^{h_A(\Im(w_{k+1}))+\sum_{j=1}^k h_A(\Im(w_j))+h_{A}(-\Im(w_j))},
			\end{align*}
			which shows the desired estimate.
		\end{proof}
		
		The following is a version of \autoref{theorem:PWS_LV} for the functions $\F(\Psi)$, $\Psi\in\P_0\LV_{c,k}(\R^n)$.
		\begin{proposition}
			\label{proposition:PWS_non-trivialDirection}
			Let $F\in\mathcal{O}_{\C^n}\widetilde{\M^2_k}$, $A\subset \R^n$ compact and convex with non-empty interior, and assume that for every $N\in\mathbb{N}$ there exists $\delta_N>0$ and $C_N\in\mathbb{N}$ such that for $w\in\Mat_{n,k+1}(\C)$ with $|w_j|\le \delta_N$ for $1\le j\le n$,
			\begin{align}
				\label{eq:estimatePWS}
				|F(w)|\le C_N(1+|w_{k+1}|)^{-N}e^{h_A(\Im (w_{k+1}))}.
			\end{align}
			Then $F=\F(\Psi)$ for a unique smooth local functional $\Psi\in \P_0\LV_k(\R^n)$ with $\locsupp\Psi\subset A$. More precisely, for any basis $\Psi_j$, $1\le j\le N_{n,k}$, of $\MAVal_k(\R^n)$ there exist functions $\phi_j\in C_c^\infty(\R^n)$ with $\supp \phi_j\subset A$ such that 
			\begin{align*}
				\Psi=\sum_{j=1}^{N_{n,k}}\phi_j\bullet \Psi_j.
			\end{align*} 
			Conversely, for any smooth local functional $\Psi\in \P_0\LV_k(\R^n)$ with $\locsupp\Psi\subset A$, the function $\F(\Psi)$ satisfies estimates of the form \eqref{eq:estimatePWS} for every $N\in\mathbb{N}$.
		\end{proposition}
		\begin{proof}
			If $F$ is a function with the given properties, then we may apply \autoref{theorem:GroebnerAlgModuleMinor} to obtain $g_j\in\mathcal{O}_{\C^n}$ such that
			\begin{align*}
				F(w)=\sum_{j=1}^{N_{n,k}}g_j(w_{k+1})Q(\Psi_j)[w_1,\dots,w_k],
			\end{align*}
			where
			\begin{align*}
				|g_j(z)|\le C\left(1+|z|\right)^{nN_{n,k}}\sup_{\zeta\in D_1(0,\dots,0,z)}|F(\zeta)|.
			\end{align*}
			The estimates for $F$ thus imply that $g_j$ satisfies the conditions of the Paley--Wiener--Schwartz \autoref{theorem:PaleyWienerSchwartz_distributions}, so $g_j=\mathcal{F}(\phi_j)$ for some $\phi_j\in C^\infty_c(\R^n)$ with $\supp\phi_j\subset A$. The smooth local functional $\Psi:=\sum_{j=1}^{N_{n,k}}\phi_j\bullet \Psi_j$ is thus supported on $A$, and from \autoref{lemma:FourierMA} we obtain
			\begin{align*}
				\F(\Psi)[w]=\sum_{j=1}^{N_{n,k}}g_j(w_{k+1})Q(\Psi_j)[w_1,\dots,w_k]=F(w).
			\end{align*}
			The uniqueness of $\Psi$ follows from the fact that $\F$ is injective.\\
			
			Conversely, if $\Psi\in \P_0\LV_k(\R^n)$ is a smooth local functional with $\locsupp\Psi\subset A$, then we may combine the estimate in \autoref{lemma:PWS_SatisfiedBySmoothVLV} with the estimate in \autoref{corollary:EstimateFourierDiagonalCoordinates} to see that $\F(\Psi)$ satisfies inequalities of the form \eqref{eq:estimatePWS} for every $N\in\mathbb{N}$. 
		\end{proof}

		We summarize the previous results in the following proposition. 
		\begin{proposition}
			\label{proposition:EquivalenceSmoothCompactSupport}
			Let $A\subset \R^n$ be compact and convex with non-empty interior. The following are equivalent for $\Psi\in  \P_0\LV_{c,k}(\R^n)$:
			\begin{enumerate}
				\item $\Psi$ is smooth and $\locsupp\Psi\subset A$.
				\item For every $N\in\mathbb{N}$ there exists a constant $C_N>0$ such that
				\begin{align*}
					|\mathcal{F}(\widehat{\GW}(\Psi))[w]|\le C_N \prod_{j=1}^k|w_j|^2 \cdot (1+|\mathrm{d}(w)|)^{-N}\exp\left(h_A\left(\Im \mathrm{d}(w)\right)\right).
				\end{align*} 
				\item For every basis $\Psi_j$, $1\le j\le N_{n,k}$, of $\MAVal_k(\R^n)$ there exist functions $\phi_j\in C^\infty_c(\R^n)$ with $\supp\phi_j\subset A$ such that 
				\begin{align*}
					\Psi=\sum_{j=1}^{N_{n,k}}\phi_j\bullet\Psi_j.
				\end{align*}
			\end{enumerate}
		\end{proposition}
		\begin{proof}
			The implication (3) $\Rightarrow$ (1) follows from \autoref{corollary:smoothnessModuleStructure} (or \cite{KnoerrPolynomiallocalfunctionals2025}*{Theorem~4.21}) and the characterization of $\locsupp\Psi$. The implication (1) $\Rightarrow$ (2) follows by combining the estimates in \autoref{corollary:GeneralEstimateFourierInTermsOfF} and \autoref{proposition:PWS_non-trivialDirection}. Finally, if $\Psi$ satisfies the estimate in (2), then it is easy to see that $\F(\Psi)$ satisfies the estimates in \autoref{proposition:PWS_non-trivialDirection} for any $N\in\mathbb{N}$. Thus \autoref{proposition:PWS_non-trivialDirection} shows that (2) implies (3).
		\end{proof}
		
		\begin{proof}[Proof of \autoref{theorem:PWS_LV}]
			The first claim follows directly from \autoref{proposition:EquivalenceSmoothCompactSupport}. If $F\in\widehat{\M^2_k}$ satisfies the estimates in Eq.~\eqref{eq:PWSconditionValuations}, then it is easy to see that the function
			\begin{align*}
				\tilde{F}(w):=\frac{k!}{(-1)^k}F\left(w_1,\dots,w_k,w_{k+1}-\sum_{j=1}^kw_j\right)
			\end{align*}
			satisfies the conditions in \autoref{proposition:PWS_non-trivialDirection}, so there exists a unique $\Psi\in \P_0\LV_{c,k}^\infty(\R^n)$ with $\locsupp \Psi\subset A$ such that $\F(\Psi)=\tilde{F}$. Unraveling the definitions, this implies $F=\mathcal{F}(\widehat{\GW}(\Psi))$, which completes the proof.
		\end{proof}

\section{Integral representation of polynomial local functionals}
	\label{section:IntegralRepresentations}
	\subsection{Smooth polynomial local functionals}
		In this section, we establish the classification of smooth polynomial local functionals in \autoref{maintheorem:IntegralRepSmooth}. We start with the case $d=0$. Recall that $\MAVal_k(\R^n)=\P_0\LV_k(\R ^n)^{tr}$ and that $N_{n,k}=\dim\P_0\LV_k(\R ^n)^{tr}$.
		\begin{theorem}
			\label{theorem:RepresentationSmoothFunctionalDuallyEpiInvariantCase}
			The following are equivalent for $\Psi\in \P_0\LV_k(\R^n)$:
			\begin{enumerate}
				\item $\Psi$ is smooth.
				\item For every basis $\Psi_j$, $1\le j\le N_{n,k}$, of $\MAVal_k(\R^n)$ there exist functions $\phi_j\in C^\infty(\R^n)$ such that 
				\begin{align*}
					\Psi=\sum_{j=1}^{N_{n,k}}\phi_j\bullet \Psi_j.
				\end{align*}
			\end{enumerate}
		\end{theorem}
		\begin{proof}
			The implication (2) $\Rightarrow$ (1) is a consequence of \autoref{corollary:smoothnessModuleStructure}. Let us show the implication (1) $\Rightarrow$ (2). It is sufficient to consider the case where $\Psi\in \P_0\LV_k(\R^n)$ is a smooth local functional of degree $0\le k\le n$. Choose a locally finite cover of $\R^n$ by a countable family $(B_m)_m$ of compact convex sets with non-empty interior and fix a smooth partition of unity $(\psi_m)_m$ subordinate to this cover. Then
			\begin{align*}
				\Psi=\sum_{m=1}^\infty \psi_m\bullet \Psi,
			\end{align*}
			where the sum converges in the compact-to-bounded topology since the supports of these functionals are locally finite. Note that $\Phi_m:=\psi_m\bullet \Psi\in\P_0\LV_{c,k}(\R^n)$ is smooth by \autoref{corollary:smoothnessModuleStructure}. Since $\locsupp\Phi_m\subset B_m$ by \autoref{corollary:supportModuleProduct}, we may apply \autoref{proposition:EquivalenceSmoothCompactSupport} to find smooth functions $\phi^m_j\in C^\infty_c(\R^n)$ with $\supp\phi^m_j\subset B_m$ such that
			\begin{align*}
				\Phi_m=\sum_{j=1}^{N_{n,k}}\phi^m_j\bullet \Psi_j.
			\end{align*}
			Then $\phi_j:=\sum_{m=1}^\infty \phi^m_j\in C^\infty(\R^n)$ is well defined since the sum is locally finite, and we obtain
			\begin{align*}
				\Psi=\sum_{m=1}^\infty\Phi_m=\sum_{j=1}^{N_{n,k}}\phi_j\bullet \Psi_j.
			\end{align*}
			The claim follows.
		\end{proof}
		 Recall that a representation of a Lie group $G$ on a locally convex vector space $F$ is called continuous if the map
		\begin{align*}
			G\times F&\rightarrow F\\
			(g,v)&\mapsto g\cdot v
		\end{align*}
		is jointly continuous. We denote by $F^{sm}\subset F$ the subspace of smooth vectors i.e. the subspace of all $v\in F$ such that the map $G\rightarrow F$, $g\mapsto g\cdot v$ is smooth. Note that we do not require $F$ to be complete, so this subspace can be trivial.
		The following well known result is stated in \cite{Aleskermultiplicativestructurecontinuous2004}*{Lemma~1.5} for continuous representations of a Lie group $G$ on a Fr\'echet space. The proof holds verbatim in the more general situation below.
		\begin{lemma}\label{lemma:smoothVectorsTensorProduct}
			Let $G$ be a Lie group and $F$ a continuous representation of $G$ on a locally convex vector space. If $S$ is a finite dimensional continuous representation of $G$, then $(F\otimes S)^{sm}=F^{sm}\otimes S^{sm}$.
		\end{lemma}
		
		\begin{proof}[Proof of \autoref{maintheorem:IntegralRepSmooth}]
			Any valuation of the form 
			\begin{align}
				\label{eq:polynomialRepSmoothCase}
				\Psi=\sum_{j=1}^{N(n,d)} \phi_j\bullet \Psi_j
			\end{align} for $\phi_j\in C^\infty(\R^n)$ and a basis $\Psi_j$ of $\P_d\LV(\R^n)^{tr}$ is smooth by \autoref{corollary:smoothnessModuleStructure}. For the converse implication, note that it is sufficient to show that any smooth polynomial local functional of degree at most $d$ admits such a representation with respect to some basis, which we establish by induction on $d$. The base case $d=0$ is covered by \autoref{theorem:RepresentationSmoothFunctionalDuallyEpiInvariantCase}, so let $d\ge 1$ and assume that the claim holds for all smooth polynomial local functionals of degree at most $d-1$. For $\Psi\in\P_d\LV(\R^n)$, let $\Psi_j\in \P_{d-j}\VConv(\R^n,\M(\R^n)\otimes \Sym^{j}(\A(n,\R)^*)_\C)$ be the unique functionals with 
			\begin{align*}
				\Psi(f+\ell)=\sum_{j=0}^d \Psi_j(f)[\ell]
			\end{align*}
			for $f\in\Conv(\R^n,\R)$, $\ell\in \A(n,\R)$, compare \autoref{proposition:translativeDecompCompatibleHomDecomp}.	Using the inverse of the Vandermonde matrix, we find constants $c_{ij}\in\R$ independent of $\Psi$ such that
			\begin{align}
				\label{eq:formulaPsi_i}
				\Psi_i(f)[\ell]=\sum_{j=0}^{d} c_{ij}\Psi(f+j\ell)
			\end{align}
			for all $f\in\Conv(\R^n,\R)$ and $\ell\in \A(n,\R)$. In particular $f\mapsto \Psi_i(f)[\ell]$ is locally determined, so the map 
			\begin{align}
				\label{eq:MapHighestOrderTerm}
				\begin{split}
				\P_d\LV(\R^n)&\rightarrow \P_0\LV(\R^n)\otimes \Sym^d(\A(n,\R)^*)_\C\\
				\Psi&\mapsto \Psi_d
				\end{split}
			\end{align}
			is well defined, and from Eq.~\eqref{eq:formulaPsi_i} we obtain that this map is continuous with respect to the compact-to-compact topology. Moreover, for $x\in\R^n$, we have
			\begin{align*}
					[\pi(x)\Psi](f+\ell;\phi)=&\Psi(f(\cdot+x)+\ell+\ell(x);\phi(\cdot+x))\\
					=&\sum_{j=0}^d \Psi_j(f(\cdot+x);\phi(\cdot+x))[\ell+\ell(x)],
			\end{align*}
			so comparing degrees, we obtain
			\begin{align*}
				[\pi(x)\Psi]_d(f;\phi)[\ell]=\Psi_d(f(\cdot+x);\phi(\cdot+x))[\ell+\ell(x)].
			\end{align*}
			If we let $\R^n$ act on $\Sym^d(\A(n,\R)^*)$ using the representation induced from the natural operation of $\R^n$ on $\A(n,\R)$ given by $\ell\mapsto \ell(\cdot+x)=\ell+\ell(x)$, this implies that the map in Eq.~\eqref{eq:MapHighestOrderTerm} is equivariant with respect to translations. Since $\Sym^d(\A(n,\R))_\C$ is a finite dimensional representation, it now follows from \autoref{lemma:smoothVectorsTensorProduct} that any smooth $\Psi\in \P_d\LV(\R^n)$ is mapped to an element of $\P_0\LV(\R^n)\otimes \Sym^d(\A(n,\R)^*)_\C$ whose components with respect to some basis of $\Sym^d(\A(n,\R)^*)_\C$ (considered as elements in $\P_0\LV(\R^n)$) are smooth (here we use that $\P_d\LV(\R^n)$ is a continuous representation of $\R^n$ with respect to the compact-to-compact topology by \autoref{proposition:AffActionSeparatelyContinuous}). From \autoref{theorem:RepresentationSmoothFunctionalDuallyEpiInvariantCase}, we thus obtain smooth functions $\phi_j\in C^\infty(\R^n)\otimes\Sym^d(\A(n,\R)^*)_\C$ such that for all $f\in\Conv(\R^n,\R)$, $\ell\in\A(n,\R)$,
			\begin{align*}
				\Psi_d(f)[\ell]=\sum_{j=1}^{N(n,0)} \left[\phi_j(\cdot)[\ell]\bullet\Psi_j\right](f),
			\end{align*}
			where $\Psi_j$, $1\le j\le N(n,0)$, denotes a basis of $\MAVal(\R^n)$. Using \autoref{theorem:ClassificationPolyTranslationInv}, we obtain $\tilde{\Psi}\in\P_d\LV(\R^n)$ such that for $f\in C^2(\R^n)\cap \Conv(\R^n,\R)$, $B\subset\R^n$ bounded Borel set,
			\begin{align*}
				\tilde{\Psi}(f;B)=\sum_{j=1}^{N(n,0)}\int_{B} \phi_j(x)[df(x),f(x)]d\Psi_j(f;x).
			\end{align*}
			Then $\tilde{\Psi}$ is given by an expression of the form in Eq.~\eqref{eq:polynomialRepSmoothCase} and is thus smooth by \autoref{corollary:smoothnessModuleStructure}. Moreover, it is easy to see that $\Psi-\tilde{\Psi}\in \P_{d-1}\LV(\R^n)$. Since $\Psi-\tilde{\Psi}$ is smooth, we may apply the induction assumption to express this difference as a sum of local functionals of the form given in Eq.~\eqref{eq:polynomialRepSmoothCase}. Thus $\Psi$ also admits a representation of the form in Eq.~\eqref{eq:polynomialRepSmoothCase}, which completes the proof.	
		\end{proof}
			
	\subsection{Continuous polynomial local functionals}
		In order to obtain \autoref{maintheorem:IntegralRepContinuousCase} from the description of smooth local functionals in \autoref{maintheorem:IntegralRepSmooth}, we will realize this decomposition using certain evaluation maps in convex polynomials. Let $\mathcal{C}\subset\Conv(\R^n,\R)$ denote the space of convex polynomials of degree at most $2$. Any such function is given by
		\begin{align*}
			q(x)=\langle x,Ax\rangle+l(x)+c
		\end{align*}
		for a positive semi-definite matrix $A\in\Sym^2(\R^n)$, $l\in (\R^n)^*$ and $c\in\R$. In particular, \autoref{theorem:ClassificationPolyTranslationInv} implies that for every $\Psi\in \P_d\LV(\R^n)^{tr}$, we have
		\begin{align*}
			\Psi(q;B)=\int_{B} P_\Psi(\langle x,Ax\rangle+l(x)+c,\langle Ax,\cdot\rangle+l,A)dx,
		\end{align*}
		for a polynomial $P_\Psi\in \Poly_d(\A(n,\R))\otimes\mathrm{M}_n$. In other words, the measure $\Psi(q)$ is absolutely continuous with respect to the Lebesgue measure with a polynomial density. We will show that suitable combinations of these evaluation maps can be used to obtain sufficiently many functionals $E:\P_d\LV(\R^n)\rightarrow \M(\R^n)$ such that
		\begin{enumerate}
			\item the functional $E$ maps $\P_d\LV(\R^n)^{tr}$ to scalar multiples of the Lebesgue measure,
			\item the functionals separate elements in $\P_d\LV(\R^n)^{tr}$.
		\end{enumerate}
		In order to makes this precise, let $W_{d,l}\subset \P_d\LV(\R^n)^{tr}$ be the subspace of all $\Psi$ such that
		\begin{align*}
			c\mapsto \Psi(f+c)
		\end{align*}
		is a polynomial of degree at most $0\le l\le d$. Consider the subspace $V_{d,l}\subset W_{d,l}^*$ of all linear maps $T:W_{d,l}\rightarrow \C$ such that there exist $N\in\mathbb{N}$, $q_j\in \mathcal{C}$, and $c_j\in \C$ for $1\le j\le N$ such that
		\begin{align}
			\label{eq:defVd}
			\sum_{j=1}^N c_j\Psi(q_j)= T(\Psi)\vol.
		\end{align}
		Note that we have a natural restriction map $V_{d,l}\mapsto V_{d,l-1}$ for $1\le l\le d$ and $V_{d,0}\mapsto V_{d-1,0}$ for $d\ge 1$. However, it is a priori not clear whether these maps are surjective, since the linear combination of evaluation maps corresponding to an element in $V_{d,l-1}$ may map elements of the larger space $W_{d,l}$ to a measure with a polynomial density. The following two lemmas show how this problem can be avoided. We start with the case $l=0$.
		\begin{lemma}\label{lemma:EvaluationSpanWd0}
			\begin{enumerate}
				\item For every $d\ge 0$ there exist $T_1,\dots,T_{M(d)}\in V_{d,0}$ with $T_j|_{W_{d-1,0}}=0$, and $\Psi_1,\dots,\Psi_{M(d)}\in W_{d,0}$, where $M(d)=\dim W_{d,0}-\dim W_{d-1,0}$, such that
				\begin{align*}
					\Psi-\sum_{j=1}^{M(d)}T_j(\Psi)\Psi_j\in W_{d-1,0}.
				\end{align*} 
				\item The natural restriction map $V_{d,0}\rightarrow V_{d-1,0}$ is surjective.
				\item $V_{d,0}=W_{d,0}^*$.
			\end{enumerate}
		\end{lemma}
		\begin{proof}
			Note that (1) implies (2) since the map
			\begin{align*}
				\Psi\mapsto \Psi-\sum_{j=1}^{M(d)}T_j(\Psi)\Psi_j
			\end{align*}
			is a projection $W_{d,0}\rightarrow W_{d-1,0}$ and the composition of any element in $V_{d-1,0}$ with this projection defines an element of $V_{d,0}$ by construction. Noting that $V_{0,0}=W_{0,0}^*$ due to the description of $\P_0\LV(\R^n)^{tr}$ in \autoref{theorem:ClassificationPolyTranslationInv}, (3) follows by induction from (1) and (2).\\
			
			Let us show (1). By \cite{KnoerrPolynomiallocalfunctionals2025}*{Lemma~5.5}, the isomorphism	$\Poly_d(\A(n,\R))\otimes\mathrm{M}_n\cong \P_d\LV(\R^n)^{tr}$ from \autoref{theorem:ClassificationPolyTranslationInv} restricts to an isomorphism 
			\begin{align}
				\label{eq:isomWd}
				\Poly_d((\R^n)^*)\otimes \mathrm{M}_n\cong W_{d,0}.
			\end{align}
			If we choose a basis $\Psi_1,\dots,\Psi_{N(n,0)}$ of $\P_0\LV(\R^n)$, we may consider $W_{d,0}$ with the basis $\Psi^\alpha_1,\dots,\Psi^\alpha_{N(n,0)}$ corresponding to the polynomials $\ell^\alpha P_{\Psi_j}\in \Poly_d((\R^n)^*)\otimes \M_n$, $\alpha\in \mathbb{N}^n$, $|\alpha|\le d$, under this isomorphism. Using the inverse of the Vandermonde matrix, we obtain coefficients $c_j$ independent of $\Psi\in W_{d,0}$ and $\ell\in(\R^n)^*$ such that the homogeneous term of order $d$ of the polynomial $\ell\mapsto \Psi(\cdot+\ell)$ is given by $\ell\mapsto \sum_{j=0}^{d}c_j\Psi(\cdot+j\ell)$. This implies for every $\alpha\in \mathbb{N}^n$ with $|\alpha|\le d$, and $1\le j\le N(n,0)$,
			\begin{align*}
				\sum_{j=0}^{d}c_j\Psi_i^\alpha(q+j\ell)=\begin{cases}
					\ell^\alpha P_{\Psi_i}(D^2q)\vol, &\text{for}~ |\alpha|=d,\\
					0, & \text{else},
				\end{cases}
			\end{align*}
			for all $q\in\mathcal{C}$ and every $\ell\in (\R^n)^*$.
			We therefore find $N>0$ and for every $|\alpha|=d$ and $1\le j\le N(n,0)$ elements $q^{\alpha,j}_i\in\mathcal{C}$ and constants $c^{\alpha,j}_i$, $1\le i\le N$,  such that
			\begin{align*}
				\sum_{|\alpha|=d}\sum_{i=0}^{N}c^{\alpha,j}_i\Psi_{j'}^{\alpha'}(q^{\alpha,j}_i)=\delta_{j,j'}\delta_{\alpha,\alpha'}\vol
			\end{align*}
			for all $1\le j'\le N(n,0)$ and $\alpha'\in \mathbb{N}^n$ with $|\alpha'|\le d$.
			In particular, the left hand side defines a map $T^\alpha_j\in V_{d,0}$ for all $|\alpha|=d$ such that 
			\begin{align*}
				\Psi-\sum_{|\alpha|=d}\sum_{j=1}^{N(n,0)}T^\alpha_j(\Psi) \Psi_j^\alpha\in W_{d-1,0}
			\end{align*}
			for all $\Psi\in W_{d,0}$. Moreover, by construction the corresponding map $W_{d,0}\rightarrow W_{d-1,0}$ is the identity on $W_{d-1,0}$, which concludes the proof of (1).
		\end{proof}
			
		\begin{lemma}\label{lemma:evaluationSpanWdl}
		\begin{enumerate}
			\item For every $d\ge 1$, $1\le l\le d$, there exist $T_1,\dots,T_{M(d,l)}\in V_{d,l}$ with $T_j|_{W_{d,l-1}}=0$, and $\Psi_1,\dots,\Psi_{M(d,l)}\in W_{d,l}$, where $M(d,l)=\dim W_{d,l}-\dim W_{d,l-1}$ such that
			\begin{align*}
				\Psi-\sum_{j=1}^{M(d,l)}T_j(\Psi)\Psi_j\in W_{d,l-1}.
			\end{align*} 
			\item The natural restriction maps $V_{d,l+1}\rightarrow V_{d,l}$ are surjective for $0\le l\le d-1$, $d\ge 1$.
			\item $V_{d,l}=W_{d,l}^*$.
		\end{enumerate}
		\end{lemma}
		\begin{proof}
			As in the proof of \autoref{lemma:EvaluationSpanWd0}, (1) implies (2), from which (3) follows by induction from the case $W_{d,0}$, which is covered by \autoref{lemma:EvaluationSpanWd0}.\\
			
			In order to see that (1) holds, note that since $c\mapsto \Psi(\cdot+c)$, $c\in\R$, is a polynomial of degree at most $l$ for every $\Psi\in W_{d,l}$, there are maps $Y_{\Psi,j}:\Conv(\R^n,\R)\rightarrow\M(\R^n)$ such that
			\begin{align*}
				\Psi(f+c)=\sum_{j=0}^l c^jY_{\Psi,j}(f)
			\end{align*}
			for every $f\in\Conv(\R^n,\R)$, $c\in\R$. Comparing coefficients in $c$, it is easy to see that $c\mapsto Y_{\Psi,j}(f+c)$ is a polynomial of degree at most $l-j$. We may use the inverse of the Vandermonde matrix to obtain coefficients $c_0,\dots,c_l$ independent of $\Psi$, $f$, and $c$, such that
			\begin{align*}
				Y_{\Psi,l}(f)=\sum_{j=0}^lc_j \Psi(f+j).
			\end{align*}
			By comparing degrees, it is not difficult to see that $Y_{\Psi,l}\in W_{d-l,0}$. Let $\Psi_1,\dots,\Psi_{\dim W_{d-l,0}}$ be a basis of $W_{d-l,0}$ and let $\Psi^l_j\in W_{d,l}$ be the element corresponding to the polynomial $c^l P_{\Psi_j}\in \Poly_d(\Aff(n,\R))\otimes\M_n$ under the isomorphism in \autoref{theorem:ClassificationPolyTranslationInv} (note that it is not difficult to see that $\Psi^l_j$ actually belongs to $W_{d,l}$). From \autoref{lemma:EvaluationSpanWd0}, we obtain elements $T_1,\dots, T_{\dim W_{d-l,0}}\in V_{d-l,0}$ that form the dual basis to $\Psi_1,\dots,\Psi_{\dim W_{d-l,0}}$. Since 
			\begin{align*}
				Y_{\Psi,l}(f)=\sum_{j=1}^{\dim W_{d-l,0}}T_j(Y_{\Psi,l})\Psi_j=\sum_{j=1}^{\dim W_{d-l,0}}T_j\left(\sum_{j=0}^lc_j \Psi(\cdot+j)\right)\Psi_j,
			\end{align*}
			it is easy to see that
			\begin{align*}
				\Psi-\sum_{j=1}^{\dim W_{d-l,0}}T_j\left(\sum_{j=0}^lc_j \Psi(\cdot+j)\right)\Psi^l_j\in W_{d,l-1},
			\end{align*}
			where the maps 
			\begin{align*}
				\Psi\mapsto T_j\left(\sum_{j=0}^lc_j \Psi(\cdot+j)\right)
			\end{align*}
			belong to $W_{d,l}$ by construction. This completes the proof.
		\end{proof}
		
		\begin{corollary}\label{corollary:evaluationMapsDualBasis}
				Let $\Psi_j$, $1\le j\le N(n,d)$, be a basis of $\P_d\LV(\R^n)^{tr}$. There exist $M_d\in\mathbb{N}$, $c_{ij}\in \R$ and $q_{ij}\in \mathcal{C}$, $1\le i,j\le M_d$, such that the maps $T^d_j:\P_d\LV(\R^n)\rightarrow\M(\R^n)$ given by
			\begin{align*}
				E_j(\Psi):=\sum_{i=1}^{M_d}c_{ij}\Psi(q_{ij})
			\end{align*}
			satisfy
			\begin{align*}
				E_j(\Psi_i)=\delta_{ij}\vol.
			\end{align*}
		\end{corollary}
		\begin{proof}
			Since $V_{d,d}=W_{d,d}^*$ by \autoref{lemma:evaluationSpanWdl}, which is the dual space of $\P_d\LV(\R^n)^{tr}$, we can realize the dual basis to $\Psi_j$, $1\le j\le N(n,d)$, by maps in $V_{d,d}$, which provides the maps $E_j$.
		\end{proof}	
		\autoref{maintheorem:IntegralRepContinuousCase} follows from the following result.
		\begin{theorem}
			\label{theorem:IntegralRepCont}
			Let $\Psi_j$, $1\le j\le N(n,d)$, be a basis of $\P_d\LV(\R^n)^{tr}$ and $E_j:\P_d\LV(\R^n)\rightarrow\M(\R^n)$ the associated maps from \autoref{corollary:evaluationMapsDualBasis}. Then for every $\Psi\in \P_d\LV(\R^n)$,
			\begin{align*}
				\Psi(f;\phi)
				=\sum_{j=1}^{N(n,d)}\int_{\R^n} \phi(x) P_{\Psi_j}(f(x),df(x),D^2f(x))dE_j(\Psi)
			\end{align*}
			for every $f\in C^2(\R^n)\cap \Conv(\R^n,\R)$, $\phi\in C_c(\R^n)$. Moreover, if $\Psi\in \P_d\LV^0(\R^n)$, then $E_j(\Psi)$ is absolutely continuous with respect to the Lebesgue measure.
		\end{theorem}
		\begin{proof}
			For any $f\in C^2(\R^n)\cap\Conv(\R^n,\R)$ and $\phi\in C_c(\R^n)$, the map
			\begin{align*}
			 	\P_d\LV(\R^n)&\rightarrow\C\\
			 	\Psi&\mapsto \Psi(f;\phi)-\sum_{j=1}^{N(n,d)}\int_{\R^n} \phi(x) P_{\Psi_j}(f(x),df(x),D^2f(x))dE_j(\Psi)
			\end{align*} 
			is continuous with respect to the compact-to-weak* topology. \autoref{corollary:evaluationMapsDualBasis} and \autoref{maintheorem:IntegralRepSmooth} show that this map vanishes on the subspace of smooth local functionals. Since smooth local functionals are dense in $\P_d\LV(\R^n)$ with respect to the compact-to-weak* topology by \autoref{proposition:densitySmoothVectors}, this map vanishes identically, which shows the desired representation formula.\\
			
			If $\Psi$ is strongly continuous, then we may approximate $\Psi$ with respect to the compact-to-bounded topology by a sequence of smooth local functionals by \autoref{proposition:densitySmoothVectors}. Obviously the maps $E_j:\P_d\LV(\R^n)\rightarrow\M(\R^n)$ are continuous if the two spaces are equipped with the compact-to-bounded and strong topology respectively. On smooth local functionals, the image of $E_j$ consists of measures that are absolutely continuous with respect to the Lebesgue measure, and the space of these measures is closed in $\M(\R^n)$ with respect to the strong topology. Thus $E_j(\Psi)$ is absolutely continuous with respect to the Lebesgue measure as well.
		\end{proof}
		
\section{Affine invariant submodules}
		\label{section:affineInvSubModule}
		In this section, we prove a stronger version of \autoref{maintheorem:affineInvariantSubmoduleDense}. We begin with a correspondence between certain submodules of $\P_d\LV(\R^n)$ and subspaces of $\P_d\LV(\R^n)^{tr}$.

		\begin{theorem}\label{theorem:CorrespondenceSubmodulesSubspaces}
			\begin{enumerate}
				\item Consider $\P_d\LV(\R^n)$ with the compact-to-weak* or 	compact-to-bounded topology. The map
					\begin{align*}
						W\mapsto W\cap \P_d\LV(\R^n)^{tr}
					\end{align*}
					establishes a $1$-to-$1$-correspondence between the closed and translation invariant $C(\R^n)$-submodules $W$ of $\P_d\LV(\R^n)$ and the subspaces of $\P_d\LV(\R^n)^{tr}$. 
				\item Consider the space $\P_d\LV^0(\R^n)$ of all strongly continuous local functionals with the compact-to-bounded topology. The map
					\begin{align*}
						W\mapsto W\cap \P_d\LV(\R^n)^{tr}
					\end{align*}
					establishes a $1$-to-$1$-correspondence between the closed and translation invariant $C(\R^n)$-submodules $W$ of $\P_d\LV^0(\R^n)$ and the subspaces of $\P_d\LV(\R^n)^{tr}$. 
			\end{enumerate}
		\end{theorem}
		\begin{proof}
			The argument is the same independent of the chosen topology, the only difference is that the closure of the space of smooth local functionals with respect to the compact-to-bounded topology is the space $\P_d\LV^0(\R^n)$, compare \cite{KnoerrPolynomiallocalfunctionals2025}*{Corollary~3.22}, which is the reason for the additional restriction in (2). Let us therefore fix one of these three topologies.\\			
			
			We claim that the inverse map is given by associating to a subspace $V\subset \P_d\LV(\R^n)^{tr}$ the subspace
			\begin{align*}
				\tilde{V}=\overline{\mathrm{span}\{\phi\bullet\Psi:\phi\in C(\R^n),\Psi\in V\}},
			\end{align*}
			where we take the closure with respect to the chosen topology. Since the group of translations acts by continuous maps on $\P_d\LV(\R^n)$ with respect to the compact-to-weak* and compact-to-compact topology by \autoref{proposition:AffActionSeparatelyContinuous} and on $\P_d\LV^0(\R^n)$ with respect to the compact-to-bounded topology by \autoref{proposition:PropertiesStronglyContVectors}, $\tilde{V}$ is translation invariant. Since the module structure is separately continuous with respect to the three topologies by \autoref{proposition:continuityModuleStructure}, $\tilde{V}$ is also a $C(\R^n)$-submodule.\\
			
			Choose a basis $\Psi_j$, $1\le j\le N(n,d)$, of $\P_d\LV(\R^n)^{tr}$ such that $\Psi_1,\dots,\Psi_N$ is a basis for $V$, and let $E_j:\P_d\LV(\R^n)\rightarrow\M(\R^n)$ be the associated maps from \autoref{corollary:evaluationMapsDualBasis}. Then $E_j(\Psi)=0$ for $j>N$ for all $\Psi\in \tilde{V}$ since $E_j$ is continuous and vanishes an a dense subset of $\tilde{V}$. In particular, $\tilde{V}\cap \P_d\LV(\R^n)^{tr}=V$.\\
			
			This implies that the maps in (1) and (2) are surjective. In order to see that they are injective, let $W$ be a closed and translation invariant $C(\R^n)$-submodule of $\P_d\LV(\R^n)$ (or $\P_d\LV^0(\R^n)$ in the second case) and set $V:=W\cap \P_d\LV(\R^n)^{tr}$. We will show that the set of smooth local functionals in $W$ coincides with the $C^\infty(\R^n)$-module generated by $V$. Since smooth local functionals are dense in $W$ by \autoref{proposition:densitySmoothVectors}, this implies $W=\tilde{V}$, which shows that $W$ is uniquely determined by $W\cap \P_d\LV(\R^n)^{tr}$ under the correspondence above.\\
			
			Let us again choose a basis $\Psi_j$, $1\le j\le N(n,d)$, of $\P_d\LV(\R^n)$ such that $\Psi_1,\dots,\Psi_N$ is a basis of $V$. Assume that there exists a smooth local functional $0\ne\Psi\in W$ that does not belong to the $C^\infty(\R^n)$-submodule generated by $V$. Since this submodule is contained in $W$ by definition, we may use \autoref{maintheorem:IntegralRepSmooth} to subtract suitable multiples of $\Psi_1,\dots,\Psi_N$ from $\Psi$ and assume without loss of generality that $\Psi$ is given by
			\begin{align*}
				\Psi=\sum_{j={N+1}}^{N(n,d)}\phi_j\bullet \Psi_j
			\end{align*}
			for $\phi_j\in C^\infty(\R^n)$, $N+1\le j\le N(n,d)$, where at least one of these functions is non-trivial. Since $W$ is a $C(\R^n)$-module, we may multiply this equation with a function $\phi\in C^\infty_c(\R^n)$. Without loss of generality we can therefore assume that the functions $\phi_j$ are compactly supported and that at least one of the constants
			\begin{align*}
				c_j:=\int_{\R^n}\phi_j(x) dx, \quad N+1\le j\le N(n,d),
			\end{align*}
			is not equal to $0$. Note that since $\Psi_j$ is translation invariant, \begin{align*}
				\pi(x)\Psi=\sum_{j=N+1}^{N(n,d)}\phi_j(\cdot-x)\bullet \Psi_j.
			\end{align*}
			For each $\epsilon>0$, cover $\R^n$ with the cubes $Z^\epsilon_{k}=\{x\in\R^n:\epsilon k_j\le x_j\le \epsilon(k_j+1)\}$, $k\in\mathbb{Z}^n$, and fix $x^\epsilon_k\in Z^\epsilon_k$ for each $k\in\mathbb{Z}^n$. Then for each $m\in\mathbb{N}$, we have
			\begin{align*}
				W\ni \Psi^\epsilon_m:=\epsilon^n\sum_{k\in\mathbb{Z}^n,|k|\le m}\pi(x^\epsilon_k)\Psi=\sum_{j=N+1}^{N(n,d)}\left(\epsilon^n\sum_{k\in\mathbb{Z}^n,|k|\le m}\phi_j(\cdot-x^\epsilon_k)\right)\bullet \Psi_j.
			\end{align*}
			The functions
			\begin{align*}
				\phi_{j,m}^\epsilon:=\epsilon^n\sum_{k\in\mathbb{Z}^n,|k|\le m}\phi_j(\cdot-x^\epsilon_k)\in C_c(\R^n)
			\end{align*}
			are given by Riemann sums, and it is easy to see that locally uniformly on $\R^n$, we have
			\begin{align*}
				\lim_{m\rightarrow\infty}\phi_{j,m^2}^{\frac{1}{m}}=\int_{\R^n}\phi_j(\cdot-x)dx=\int_{\R^n}\phi_j(x)dx=c_j.
			\end{align*}
			Since the module action of $C(\R^n)$ is continuous with respect to the compact-to-bounded topology by \autoref{proposition:continuityModuleStructure}, we see that
			\begin{align*}
				\lim_{m\rightarrow\infty}\Psi^{\frac{1}{m}}_{m^2}=\sum_{j=N+1}^{N(n,d)}c_j\Psi_j
			\end{align*}
			in the compact-to-bounded topology. In particular, the sequence converges with respect to the compact-to-weak* and compact-to-compact topology, so the limit belongs to $W$. Since we assume that one of the constants $c_j$ is non-trivial, $\sum_{j=N+1}^{N(n,d)}c_j\Psi_j\notin V$ by our choice of basis, which is a contradiction. Thus every smooth local functional in $W$ belongs to the $C^\infty(\R^n)$-submodule generated by $V$.
		\end{proof}
		
		In order to obtain \autoref{maintheorem:affineInvariantSubmoduleDense} from this description, we need the following result for $\P_0\LV_k(\R^n)^{tr}=\MAVal_k(\R^n)$.
		\begin{theorem}[\cite{KnoerrMongeAmpereoperators2024} Theorem 1.3]
			\label{theorem:MAVal_Irreducible}
			$\MAVal_k(\R^n)$ is an irreducible representation of $\GL(n,\R)$.
		\end{theorem}
		We will establish a slightly stronger version of \autoref{maintheorem:affineInvariantSubmoduleDense}, which requires the following additional regularity property of strongly continuous local functionals with respect to the the action of the affine group.
		\begin{proposition}\label{proposition:stronglyContAffCont}
			The map
			\begin{align}
				\label{eq:affContVectors}
				\begin{split}
					\Aff(n,\R)\times \P_d\LV^0(\R^n)&\rightarrow\P_d\LV^0(\R^n)\\
					(g,\Psi)&\mapsto \pi(g)\Psi
				\end{split}
			\end{align}
			is continuous with respect to the compact-to-bounded topology.
		\end{proposition}
		\begin{proof}
			Let $W\subset \P_d\LV^0(\R^n)$ denote the subspace of all $\Psi$ such that the map $g\mapsto \pi(g)\Psi$ is continuous. It follows from \cite{KnoerrPolynomiallocalfunctionals2025}*{Proposition~3.19} that $W\subset \P_d\LV^0(\R^n)$ is closed. Moreover \cite{KnoerrPolynomiallocalfunctionals2025}*{Corollary 3.20} shows that the restriction of the map in Eq.~\eqref{eq:affContVectors} to $W$ is continuous. Combining \autoref{corollary:smoothIsAffineSmoothPreliminary} with \autoref{maintheorem:IntegralRepSmooth}, we see that $W$ contains the space of smooth local functionals. Since these are dense in $\P_d\LV^0(\R^n)$ by \autoref{proposition:densitySmoothVectors}, this implies $W=\P_d\LV^0(\R^n)$, which shows the claim. 
		\end{proof}

		\autoref{maintheorem:affineInvariantSubmoduleDense} follows from the following result.
		\begin{theorem}
			Let $0\ne W\subset \P_0\LV_k(\R^n)$ be an affine invariant $C(\R^n)$-submodule. Then the following holds:
			\begin{enumerate}
				\item $W$ is dense in $\P_0\LV_k(\R^n)$ with respect to the compact-to-weak* and compact-to-compact topology.
				\item If $W\subset \P_0\LV_k^0(\R^n)$, then $W$ is sequentially dense in $\P_0\LV_k^0(\R^n)$ with respect to the compact-to-bounded topology.
			\end{enumerate}
		\end{theorem}
		\begin{proof}
			For the first claim, note that the closure $\overline{W}\subset \P_0\LV_k(\R^n)$ with respect to either of these topologies has the same properties due to \autoref{proposition:AffActionSeparatelyContinuous}, \autoref{proposition:stronglyContAffCont}, and \autoref{proposition:continuityModuleStructure}. In the second case, $\overline{W}\subset \P_0\LV^0_k(\R^n)$ due to \autoref{proposition:PropertiesStronglyContVectors}. In all cases, $\overline{W}\cap \P_0\LV_k(\R^n)^{tr}$ is a non-trivial $\GL(n,\R)$-invariant subspace and thus coincides with $\P_0\LV_k(\R^n)^{tr}$ by \autoref{theorem:MAVal_Irreducible}.
			Thus \autoref{theorem:CorrespondenceSubmodulesSubspaces} shows that $\overline{W}=\P_0\LV_k(\R^n)$ in the first case, and $\overline{W}=\P_0\LV_k^0(\R^n)$ in the second case. In particular, $W\subset \P_d\LV_k(\R^n)$ is dense with respect to the compact-to-weak* and compact-to-compact topology. In the second case, $W\subset \P_0\LV_k^0(\R^n)$ is dense with respect to the compact-to-bounded topology, however, since this topology is Fr\'echet by \cite{KnoerrPolynomiallocalfunctionals2025}*{Theorem~F}, $W$ is already sequentially dense.
		\end{proof}
		
		\begin{proof}[Proof of \autoref{maincorollary:ExampleDenseSubmodule}]
			This result follows directly from \autoref{maintheorem:affineInvariantSubmoduleDense} since the $C(\R^n)$-module generated by the local functionals 
			\begin{align*}
				\MA(\cdot[k],f_1,\dots,f_{n-k}),\quad f_1,\dots,f_{n-k}\in\mathcal{F}
			\end{align*}
			is affine invariant.
		\end{proof}
		Note that \autoref{maincorollary:ExampleDenseSubmodule} applies in particular to the following sets of convex functions:
			\begin{itemize}
			\item $\{h_{\Delta_j}(\cdot-y):\Delta_j~j\text{-dimensional simplex},~y\in\R^n\}$, $1\le j\le n$,
			\item $\{h_\mathcal{E}(\cdot-y):\mathcal{E}~j\text{-dimensional ellipsoid},~y\in\R^n\}$, $1\le j\le n$,
			\item $\{h_\mathcal{P}(\cdot-y):\mathcal{P}~j\text{-dimensional parallelotope},~y\in\R^n\}$, $1\le j\le n$,
			\item $\{\exp(\langle y,\cdot\rangle):~y\in\R^n\}$,
			\item $\{q: q~ \text{convex polynomial of degree less or equal}~j\}$, $j\ge 2$.
		\end{itemize}

\bibliographystyle{abbrv}
\bibliography{../../library/library.bib}

\Addresses
\end{document}